\documentclass[11pt]{article}

\usepackage{amsfonts,color}
\usepackage{amsmath,amssymb}
\usepackage{epsfig}
\usepackage{lscape}
\usepackage{amssymb}
\usepackage{graphicx}
\usepackage[utf8]{inputenc}	
\usepackage{amsthm,amssymb}
\usepackage{amsfonts}
\usepackage[english]{babel}
\usepackage{dsfont}
\usepackage{bbm}
\usepackage[T1]{fontenc}
\usepackage{colonequals}
\usepackage{geometry}
\usepackage{mathtools}
\usepackage{csquotes}
\usepackage{url}
\usepackage{lipsum}
\usepackage{array}
\usepackage[utf8]{inputenc}

\usepackage{amsmath,amssymb,amsfonts,mathtools}
\usepackage{mathrsfs}
\usepackage{bm}
\usepackage{enumitem}
\usepackage{bbm}

 \usepackage[percent]{overpic}

\usepackage{color}
\usepackage{amsmath}
\usepackage{graphicx}
\usepackage{footnote}
\usepackage{pgfplots}
\pgfplotsset{compat=1.18}
\usepackage{xcolor}
\usepackage{amsfonts,color}
\usepackage{amsmath,amssymb}
\usepackage{mathtools}
\usepackage{epsfig}
\usepackage{amssymb} 
\usepackage{mathtools}
\usepackage{amsthm,wasysym}
\usepackage[english]{babel}
\usepackage{cancel}
\usepackage{colonequals}
\makeatletter
\usepackage{float}
\usepackage{wrapfig}
\usepackage{nicefrac,enumitem}
\usepackage{hyperref}
\hypersetup{hidelinks}

\usepackage{fixltx2e}
\usepackage{tikz}
\usetikzlibrary{arrows.meta,positioning}
\restylefloat{figure}

\makeatletter
\let\@fnsymbol\@arabic
\makeatother
\usepackage{caption}
\usepackage{float}

\usepackage{bbm}
\newcommand{\id}{{\boldsymbol{\mathbbm{1}}}}

\newcommand{\tr}{{\rm tr}}
\newcommand{\dev}{{\rm dev}}
\newcommand{\sym}{{\rm sym}}
\newcommand{\skw}{{\rm skew}}

\newcommand{\Curl}{{\rm Curl}}

\def\dd{\displaystyle}

\newtheorem{theorem}{Theorem}[section]

\newtheorem{proposition}[theorem]{Proposition}

\newtheorem{counterexample}[theorem]{Counterexample}

\newcommand{\todo}[1][]{TODO\ifx&#1&\else: #1\fi}

\def\dd{\displaystyle}

\usepackage{multicol}
\usepackage[font=small]{caption}
\newcommand{\citet}[2][]{\citeauthor{#2} \cite[#1]{#2}}
\RequirePackage{amsmath}	
\RequirePackage{mathtools}	
\newcommand{\R}{\mathbb{R}}

\DeclareMathOperator{\Cof}{Cof}

\newcommand{\pdd}[3][]{\frac{\partial\ifx&#1&\else^{#1}\fi #2}{\partial #3}}

\DeclareMathOperator{\@macros@div}{div}
\renewcommand{\div}{\@macros@div}

\providecommand{\availableaturl}[2][]{%
	available at \url{#2}%
}

\makeatletter%
\@ifpackageloaded{hyperref}{}{}%
\makeatother%

\usepackage{color}
\usepackage{amsmath}

\usepackage{footnote}
\usepackage{amsfonts,color}
\usepackage{amsmath,amssymb}
\usepackage{amssymb} 
\usepackage{mathtools}
\usepackage{amsthm,wasysym}
\usepackage{framed}
\usepackage{cancel}
\usepackage{graphicx}
\usepackage{caption}
\usepackage{subcaption}
\usepackage{rotating}

\usepackage[english]{babel}
\usepackage[utf8]{inputenc}
\makeatletter
\usepackage{float}
\usepackage{wrapfig}
\restylefloat{figure}

\makeatletter
\let\@fnsymbol\@arabic
\makeatother

\usepackage{bbm}

\def\dd{\displaystyle}

\def\barr{\begin{array}}
	\usepackage{lscape}

	\def\earr{\end{array}}
\def\bec#1{\begin{equation}\label{#1}}
	\def\becn{\begin{equation*}}
		\def\endec{\end{equation}}
	\def\endecn{\end{equation*}}
\def\dd{\displaystyle}
\def\bfm#1{\mbox{\boldmat}}

\renewcommand{\dd}{\displaystyle}

\usepackage{enumitem}

\usepackage{pifont}

\tikzset{
	state/.style={
		rectangle,
		rounded corners,
		draw=black, very thick,
		minimum height=2em,
		inner sep=6pt,
		text centered,
	}
}
\newcommand{\imin}[1][\lambda]{m{\ifx&#1&\else(#1)\fi}}
\newcommand{\imax}[1][\lambda]{M{\ifx&#1&\else(#1)\fi}}
\newcommand{\iset}[1][\lambda]{J{\ifx&#1&\else(#1)\fi}}
\allowdisplaybreaks[1]

\usepackage{titletoc}

\usepackage{tikz}
\usepackage{pgfplots}
\usepackage{subcaption}
\usetikzlibrary{arrows.meta,calc,positioning,patterns}
\pgfplotsset{compat=1.18}

\def\aval{0.55}
\pgfmathsetmacro{\detboundary}{(1-\aval)^2}

\usepackage{tikz}
\usepackage{pgfplots}
\usepackage{subcaption}

\usetikzlibrary{arrows.meta,calc,positioning}
\pgfplotsset{compat=1.18}

\def\aval{0.55}
\usepackage{hyperref}

\begin{document}

\title{Polyconvexity for Cosserat nonlinear elasticity and nonlinear couple--stress theory}

\author{Diana Ciotir\thanks{Diana Ciotir, Department of Mathematics, Alexandru Ioan Cuza University of Ia\c si, Blvd. Carol I, no. 11, 700506 Ia\c si, Romania, email: diaciotir@gmail.com.} \quad and \quad
Ionel-Dumitrel Ghiba\thanks{Corresponding author: Ionel-Dumitrel Ghiba, Department of Mathematics, Alexandru Ioan Cuza University of Ia\c si, Blvd. Carol I, no. 11, 700506 Ia\c si, Romania; and Octav Mayer Institute of Mathematics of the Romanian Academy, Ia\c si Branch, 700505 Ia\c si, Romania; email: dumitrel.ghiba@uaic.ro.} \quad and \quad
Patrizio Neff\thanks{Patrizio Neff, Chair for Nonlinear Analysis and Modelling, Faculty of Mathematics, University of Duisburg--Essen, Thea-Leymann-Str. 9, 45127 Essen, Germany; email: patrizio.neff@uni-due.de.}}

\maketitle

\begin{abstract}

We study a class of nonlinear elastic energies whose constitutive structure is naturally expressed in terms of a stretch variable but is not directly covered by the standard polyconvex theory formulated in the deformation gradient. The problem is lifted by introducing an independent microrotation $ \overline R \in {\rm SO}(3)$ and the relative stretch $ \overline U = \overline R ^T{\rm D}\varphi$. The curvature variable $  \overline R ^T\operatorname{Curl}  \overline R $ controls the full first-order variation of the rotation field and yields strong compactness of minimizing sequences.  The identities
$\operatorname{Cof}( \overline R^T{\rm D}\varphi)
 = \overline R^T\operatorname{Cof}{\rm D}\varphi$
and
$\det( \overline R^T{\rm D}\varphi)=\det{\rm D}\varphi$
then allow the minors of the lifted stretches to be identified through the
weak continuity of the corresponding minors of the deformation gradients.

We prove existence in two regimes. The first allows convex dependence on $(  \overline U ,\operatorname{Cof}  \overline U ,\det \overline U )$ and assumes separate cofactor coercivity. The second depends only on $( \overline U ,\det \overline U )$ and requires no independent cofactor bound. For the constrained model, the weakly closed condition $ \overline R ^T{\rm D}\varphi\in\operatorname{Sym}^{+}(3)$, together with $\det{\rm D}\varphi>0$, implies $ \overline R =R:=\operatorname{polar}({\rm D}\varphi)$ for every admissible pair. The lifted problem is therefore equivalent to a deformation problem containing the curvature of the polar factor. The resulting model is a selectively rotationally regularized nonlinear
elastic model of couple--stress type, rather than a purely first-gradient
Biot model.

\end{abstract}


\medskip


\begin{footnotesize}
\tableofcontents
\end{footnotesize}

\section{Introduction}

A central difficulty in nonlinear elasticity is the tension between mechanically natural strain measures and the structural hypotheses required by the direct method of the calculus of variations. For a deformation $\varphi:\Omega\to\mathbb R^3$, the classical polyconvex existence theory is formulated in terms of the deformation gradient $F={\rm D}\varphi$ and its minors $(F,\operatorname{Cof}F,\det F)$; see, in particular, \cite{Ball1976,Ball1977}. Polyconvexity is a convenient sufficient condition for weak lower semicontinuity, while a singular dependence on $\det F$ may enforce local orientation preservation by assigning infinite energy to $\det F\leq0$ and by producing blow-up as $\det F\to0^+$.

Several mechanically transparent energies are instead naturally written in terms of the right stretch
$
U=\sqrt{F^TF}\in {\rm Sym}^{++}(3).
$
The simplest example is the Biot strain $U-\id_3$. However, weak convergence of $F_k:={\rm D}\varphi_k$ does not imply useful weak convergence of
$
U_{k}=\sqrt{F_k^TF_k}.
$
Therefore, convexity or polyconvexity formulated directly in the {\it classical stretch} $U$ does not by itself provide the compactness mechanism needed in a first-gradient problem for $\varphi$ alone.

The present paper uses a Cosserat lifting of this difficulty. We introduce an independent microrotation $ \overline R :\Omega\to {\rm SO}(3)$. The {\it relative stretch} \footnote{In the entire paper we use the overline for $ \overline R$ and $ \overline U$ in order to mark that these matrices are not those given by the polar decomposition of $F={\rm D} \varphi$.}
$
 \overline U= \overline R^T{\rm D}\varphi\in \mathbb{R}^{3\times 3}  \  (\not \in {\rm Sym}(3), \ \text{in general}),
$
is the first Cosserat deformation tensor, while
$
\alpha= \overline R^T\operatorname{Curl} \overline R\in \mathbb{R}^{3\times 3}  
$
is a standard Cosserat \cite{Cosserat1896,Cosserat09,Cosserat09b} {\it curvature measure}.  { Integrability conditions relating the first and second Cosserat deformation
tensors were studied in \cite{LankeitNeffOsterbrink2017}. These relations
concern the compatibility between the two Cosserat tensors when the relative
deformation is invertible. In particular, for sufficiently regular fields,
$\Curl \,\overline U$ can be expressed algebraically in terms of $\overline  U$ and the rotational
curvature tensor. Moreover, the energy considered there already contains an
explicit volumetric contribution depending on $\det \overline U$. These results, however, do not provide the orientation-preserving weak
closure and the polar-factor recovery established below. Moreover, control
of $\Curl\, \overline U$ should not be confused with control of the full gradient
${\rm D}\overline U$, and therefore does not yield control of the full second gradient
${\rm D}^2\varphi$.}

For ${\rm SO}(3)$-valued tensor fields which belong to ${\rm W}^{1,q}(\Omega)$, $q\geq 2$, control of the row-wise Curl controls the full gradient of the rotation; this follows in the smooth setting from the Curl--Grad identity of Neff and M\"unch \cite{Neff_curl06} and since  smooth ${\rm SO}(3)$-valued maps are 
strongly dense in
$
{\rm W}^{1,q}(\Omega;{\rm SO}(3))
$
for every \(q\geq 2\), see  \cite{bousquet2017density,detaille2026complete} and \cite{brezis2015density,bethuel1991approximation}. Hence, for $\overline R\in  {\rm W}^{1,q}(\Omega; {\rm SO}(3))$, $q\geq 2$,  a bounded curvature energy yields weak ${\rm W}^{1,q}$-compactness and strong ${\rm L}^t$-compactness of $ \overline R _k$ for every finite $t$. Combined with weak convergence of ${\rm D}\varphi_k$, this gives
$
 \overline  R _k^T{\rm D}\varphi_k\rightharpoonup
  \overline R ^T{\rm D}\varphi.
$
The same strong--weak mechanism applies to the cofactor variable, since $\Cof \overline  R _k^T{\rm D}\varphi_k=\overline  R _k^T\Cof {\rm D}\varphi_k$, and to the
curvature strain. Direct-method existence results for geometrically nonlinear
Cosserat and, more generally,  micromorphic models with an independent microrotation were
established, among others, in
\cite{Neff_Habil04,NeffMicromorphic2006,NeffBirsan13,
neff2014existence,LankeitNeffOsterbrink2017}.

The first existence theorem for micromorphic nonlinear materials, including Cosserat nonlinear elasticity, was obtained by Neff \cite{Neff_Habil04,NeffMicromorphic2006}. Further existence and lower-semicontinuity results for nonlinear micropolar energies with independent deformation and microrotation fields were established by Tamba\v{c}a and Vel\v{c}i\'c \cite{TambacaVelcicExistence2010,TambacaVelcicSemicontinuity2010}. Their polynomial growth assumptions neither enforce $\det{\rm D}\varphi>0$ a.e. nor allow singular blow-up as $\det{\rm D}\varphi\to0^+$; consequently, they do not recover the microrotation as $\operatorname{polar}({\rm D}\varphi)$ or obtain the constraint Cosserat problem (couple--stress model \cite{Toupin62,sky2026cosserat}) considered here. Mariano and Modica developed a broader multifield existence theory in which the macroscopic deformation is a weak diffeomorphism, singular volumetric energies are allowed, and the density is polyconvex in the minors of ${\rm D}\varphi$ and convex in the full gradient of the substructural descriptor; they also discuss a Cartesian-current formulation involving all minors, including mixed minors, of the combined gradient $({\rm D}\varphi,{\rm D} \overline R)$ \cite{Mariano08a}. Their assumptions, however, do not require a convex representation in the specific objective variables
$
 \overline R^{T}F,$ $ \operatorname{Cof}( \overline R^{T}F),$ $
\det ( \overline R^{T}F),$ $ \overline R^{T}\operatorname{Curl} \overline R,
$
nor do they impose the closed constraint $ \overline R^{T}{\rm D}\varphi\in\operatorname{Sym}^{+}(3)$, recover the polar factor, or cover the reduced regime without independent cofactor coercivity. In our paper, the deliberate use of $\overline R^{T}\operatorname{Curl} \overline R$ instead of ${\rm D} \overline R$ still yields the required compactness for ${\rm SO}(3)$-valued fields and is particularly suited to future dimensional reductions towards Cosserat plate, shell, and rod models.
 
Complementary necessary conditions for minimizers in nonlinear Cosserat elasticity have also been investigated in terms of quasiconvexity, rank-one convexity and Legendre--Hadamard conditions; see, for instance, \cite{ShiraniSteigmann2022,Birsan2025Quasiconvexity}.

The basic lifting idea underlying the present approach is already contained
in Neff's habilitation thesis and in the subsequent finite-strain
micromorphic existence theory
\cite{Neff_Habil04,NeffMicromorphic2006}.  In these works, an independent microrotation or microdistortion and the
associated non-symmetric relative deformation are introduced in the
variational analysis of geometrically nonlinear generalized continua, and
the curvature contribution provides compactness for the coupled fields.
Thus, the use of an independent microrotation as an analytical lifting
device is not claimed here to be entirely new.

However, the argument in \cite{Neff_Habil04} is not formulated in terms of polyconvexity with respect to the lifted variables
$
\bigl(  \overline U ,\operatorname{Cof}  \overline U ,
\det  \overline U \bigr),
$
nor is the lifting mechanism isolated there as a general route from Cosserat models back to classical nonlinear elasticity. Some of the relevant ideas remain implicit or are distributed among different parts of the analysis and were not subsequently developed into the precise variational framework considered here.
The purpose of the present paper is therefore to extract this mechanism, formulate it systematically and provide a complete self-contained argument. In particular, we clarify the weak convergence of the lifted stretch and its minors, the weak sequential closedness of the positive-semidefinite constraint, the preservation of orientation, and the recovery of the polar factor. This makes it possible to formulate a lifted polyconvexity condition and to apply the resulting existence theory to broad classes of stretch-based energies, including regularized Biot-type energies, for which classical polyconvexity in the deformation gradient is not directly available.

The distinction from the general micropolar existence theory is also
essential. In the latter, the deformation and microrotation remain
independent unknowns, and no conclusion of the form
$
 \overline R=R:=\operatorname{polar}({\rm D}\varphi)
$
is obtained. By contrast, the present constrained formulation combines the
weakly closed condition\footnote{Here, $\operatorname{Sym}^{++}(3)$ denotes the set of positive-definite matrices, while $\operatorname{Sym}^{+}(3)=  \overline{\operatorname{Sym}^{++}(3)}$ denotes the set of positive-semidefinite matrices.}
$
 \overline R^T{\rm D}\varphi\in\operatorname{Sym}^{+}(3)
$
with the orientation-preserving condition
$
\det{\rm D}\varphi>0.
$
This yields $ \overline R=R:=\operatorname{polar}({\rm D}\varphi)$ for every admissible
pair and allows the independent Cosserat problem to be transferred
bijectively to a Cosserat constraint  problem (couple--stress model \cite{Toupin62,sky2026cosserat}).

It is important to state precisely what is and is not obtained. After polar recovery, the energy still contains the curvature of $\operatorname{polar}({\rm D}\varphi)$. Thus, the result is not an existence theorem for the non-regularized first-gradient Biot functional. It is an existence theorem for a model with selective rotational regularization: the variation of the polar rotation is controlled, whereas the variation of the stretch tensor is not. In particular, the curvature term does not yield control of the full second gradient $\mathrm D^2\varphi$. Moreover, the condition $\det{\rm D}\varphi>0$ is a local orientation condition and does not by itself imply global injectivity or exclude global interpenetration.

The choice of the curvature variable $ \overline R^{T}\operatorname{Curl}\overline R$, rather than the full gradient ${\rm D} \overline R$, is deliberate. It preserves the intrinsic Cosserat structure of the model and provides an objective measure of the spatial incompatibility of the microrotation \cite{LankeitNeffOsterbrink2017}. Here, for $\overline R\in {\rm W}^{1,q}(\Omega)$, the Curl operator is applied row-wise. 

The paper treats two complementary growth regimes. In the cofactor-dependent regime, the density is convex in $(  \overline U ,\operatorname{Cof}  \overline U ,\det  \overline U )$ and the energy controls the cofactor in ${\rm L}^r(\Omega)$. In the reduced regime, the density depends only on $(  \overline U ,\det  \overline U )$ and no separate cofactor coercivity is imposed. The main analytical point in the reduced case is that, in dimension three, $p>9/4$ allows one to combine the weak ${\rm L}^{p/2}$-convergence of the cofactors with the compact Sobolev convergence of the deformations in the Piola representation of the determinant.

The framework is more general than the uniformly convex quadratic Cosserat models considered in parts of the earlier literature \cite{Neff_Habil04,NeffBirsan13,Neff_plate04_cmt,Neff_Chelminski_ifb07}, because the stretch energy may have nonlinear polyconvex growth and a determinant singularity\footnote{Therefore, we mention that  the intention of the present paper is not to consider the following form of the energy functional \cite{Neff_Habil04}	\begin{align}
		\dd\int_{\Omega }\Big\{&
       \mu 
        \left\|\sym(
       \overline{R}^T{\rm D}\varphi
        -
        \id_3)
        \right\|^2+ \mu_{\rm c}
        \left\|\skw(
       \overline{R}^T{\rm D}\varphi
        -
        \id_3)
        \right\|^2
        +{\rm L}_{\rm c}^q \,\|\overline R^T\Curl \overline R\|^q        \Big\}{\rm d}{\rm x},\notag	\end{align}
and to  discuss the influence of the values of $\mu_{\rm c}$ and ${\rm L}_{\rm c}$ on the existence of the solution, see, e.g., \cite{Neff_Habil04,NeffBirsan13,Neff_plate04_cmt,Neff_Chelminski_ifb07}. In this respect we only mention that for $\mu_{\rm c}=0$ the energy no longer provides local control of $\overline{U}= \overline{R}^T{\rm D}\varphi$.}. It is also complementary to results that allow quadratic energies not convex in the full relative stretch. The present paper does not aim to revisit the dependence of existence on the Cosserat couple modulus; its purpose is instead to develop the lifted polyconvex framework and the polar-recovery mechanism.

 The role of the curvature term is both mechanical and analytical: it describes local rotational microstructure and supplies the strong compactness needed for the lifted direct method. The resulting formulation may be regarded as a Cosserat relaxation of a polar-curvature model in nonlinear elasticity.

To state the main result, let $\Omega\subset\mathbb R^3$ be a bounded connected Lipschitz domain and let $\Gamma_d\subset\partial\Omega$ be a nonempty relatively open set with positive surface measure. As we already mentioned, in  $
{\rm W}^{1,q}(\Omega;{\rm SO}(3))
$,  \(q\geq 2\), by controlling the row-wise Curl it means that the full gradient of the rotation is controlled, i.e., for all $\overline R\in {\rm W}^{1,q}(\Omega;{\rm SO}(3))$, it follows \begin{align}\label{lem:rotation-curl-control}
\|{\rm D} \overline R\|_{{\rm L}^q(\Omega)}
\leq
C\|\operatorname{Curl}\overline R\|_{{\rm L}^q(\Omega)},
\end{align}
where $C$ depends only on the dimension. 

We consider the following two sets of assumptions. All convexity statements below refer to extended-valued functions.

\medskip
\noindent\textbf{Cofactor-dependent regime.}
\begin{enumerate}
\item[H1)] The density has the additive form
\begin{align}
W(x,Y,K)
=
W_{\rm stretch}(x,Y,\operatorname{Cof}Y,\det Y)
+
W_{\rm curv}(x,K).
\end{align}

\item[H2)] For almost every $x\in\Omega$, the mapping
\begin{align}
(Y,Z,d,K)
\longmapsto
W_{\rm stretch}(x,Y,Z,d)+W_{\rm curv}(x,K)
\end{align}
is proper and convex on
$
\mathbb R^{3\times3}\times\mathbb R^{3\times3}\times(0,\infty)\times\mathbb R^{3\times3}.
$

\item[H3)] The extension of the preceding mapping by the value $+\infty$ for $d\leq0$ is lower semicontinuous in $(Y,Z,d,K)$ for almost every $x$, and measurable in $x$ for every fixed $(Y,Z,d,K)$. Equivalently, it is a convex normal integrand.

\item[H4)] There exist $c_1>0$ and $c_2\geq0$ such that\footnote{Here and throughout the paper, $|\cdot|$ denotes the Frobenius norm for matrices and the Euclidean norm for vectors.}
\begin{align}
W_{\rm stretch}(x,Y,Z,d)+W_{\rm curv}(x,K)
\geq
c_1\bigl(|Y|^p+|Z|^r+|d|^s+|K|^q\bigr)-c_2
\end{align}
for almost every $x$, all $Y,Z,K\in\mathbb R^{3\times3}$ and all $d>0$, where
$
p\geq2,$ $ r>1,$ $s>1,$ $ q\geq 2,
$ \text{and} $
\frac1p+\frac1r<\frac43.
$
The strict inequality is used below to obtain the compact embedding
needed in the Piola representation of the determinant. The critical
equality case is not considered in the present paper.
\end{enumerate}

\medskip
\noindent\textbf{Reduced determinant-dependent regime.}
\begin{enumerate}
\item[H1')] The density has the form
\begin{align}
W(x,Y,K)
=
W_{\rm stretch}(x,Y,\det Y)
+
W_{\rm curv}(x,K).
\end{align}

\item[H2')] For almost every $x\in\Omega$, the mapping
\begin{align}
(Y,d,K)
\longmapsto
W_{\rm stretch}(x,Y,d)+W_{\rm curv}(x,K)
\end{align}
is proper and convex on
$
\mathbb R^{3\times3}\times(0,\infty)\times\mathbb R^{3\times3}.
$

\item[H3')] Its extension by $+\infty$ for $d\leq0$ is lower semicontinuous in $(Y,d,K)$ for almost every $x$ and measurable in $x$ for every fixed $(Y,d,K)$; hence it is a convex normal integrand.

\item[H4')] There exist $c_1>0$ and $c_2\geq0$ such that
\begin{align}
W_{\rm stretch}(x,Y,d)+W_{\rm curv}(x,K)
\geq
c_1\bigl(|Y|^p+|d|^s+|K|^q\bigr)-c_2
\end{align}
for almost every $x$, all $Y,K\in\mathbb R^{3\times3}$ and all $d>0$, where
$
p>\frac94,\  s>1,\ q\geq 2.
$
\end{enumerate}

\medskip

Let $\varphi^*\in {\rm W}^{1-1/p,p}(\Gamma_d;\mathbb R^3)$ and define
\begin{align}
\mathcal A_{\rm relax}
=
\Bigl\{
(\varphi, \overline R )&\in
{\rm W}^{1,p}(\Omega;\mathbb R^3)
\times {\rm W}^{1,q}(\Omega;{\rm SO}(3))  \,|\,
\operatorname{Tr}_{\Gamma_d}\varphi=\varphi^*,
\notag\\[-1mm]
&  \overline R ^T{\rm D}\varphi\in\operatorname{Sym}^{+}(3)\ \text{a.e.},
\quad
\det{\rm D}\varphi>0\ \text{a.e.}
\Bigr\}.
\end{align}
For $(\varphi,  \overline R )\in\mathcal A_{\rm relax}$, set
\begin{align}
\mathcal I(\varphi,  \overline R )
:=
\int_\Omega
W\bigl(x,  \overline R ^T{\rm D}\varphi,
  \overline R ^T\operatorname{Curl}  \overline R \bigr)\,\mathrm dx.
\end{align}
Loads with the usual weak continuity and growth assumptions may be added without changing the compactness argument.

\begin{theorem}[Main existence result]
\label{thm:elastic1}
Assume that $\mathcal A_{\rm relax}$ contains at least one pair of finite energy. If either H1)--H4) or H1')--H4') holds, then $\mathcal I$ admits a minimizer in $\mathcal A_{\rm relax}$. Every admissible pair, and therefore every minimizer, satisfies
\begin{align}
  \overline  R =R:
=
\operatorname{polar}({\rm D}\varphi),
\quad  i.e., \quad 
 \overline  R ^T{\rm D}\varphi
=
\sqrt{({\rm D}\varphi)^T{\rm D}\varphi}
\quad\text{a.e. in }\Omega.
\end{align}
Consequently, the lifted problem is equivalent to the polar-curvature problem stated in Proposition~\ref{prop:polar-equivalence} below.
\end{theorem}

The formulation of the problem and the proof of the main result are
organized as follows. Section~\ref{s1} explains the lifted and polar formulations and the necessity of the positive-semidefinite interior constraint. Section~\ref{s2} establishes the coercivity and rotation-compactness mechanisms. Section~\ref{s3} proves existence for the auxiliary lifted Cosserat problem in both growth regimes. Section~\ref{s4} proves weak closedness of the constrained class and the polar equivalence. Section~\ref{s6} presents representative energy classes, while further admissible examples and relevant boundary cases are collected in Appendix~\ref{app:further-examples}.

\section{Classical relaxed models vs. lifted Cosserat models }
\label{s1}\setcounter{equation}{0}

Throughout the paper, $\Omega\subset\mathbb R^3$ is a bounded connected Lipschitz domain and $\Gamma_d\subset\partial\Omega$ is relatively open with positive surface measure. A generic point of \(\Omega\) is denoted by
$
        x=(x_1,x_2,x_3),
$
with respect to the fixed canonical basis
$
        \{e_1,e_2,e_3\}
$
of \(\mathbb R^3\). The deformation of the body is represented by a mapping
$
        \varphi:\Omega\to\mathbb R^3 .
$
Throughout the paper, \(\Omega\) is not merely regarded as a geometric set, but
as the reference placement of an elastic material. 
The deformation gradient is denoted by
$
F:={\rm D}\varphi,
$
and local orientation preservation means $\det F>0$ almost everywhere. This condition is local; no global injectivity assertion is made. With the convention
$
        {\rm D}\varphi
        =
        \big(
        \partial_{x_1}\varphi\,
        \big|\,
        \partial_{x_2}\varphi\,
        \big|\,
        \partial_{x_3}\varphi
        \big)
        \in\mathbb R^{3\times3},
$
the columns of \(F\) are the partial derivatives of the deformation with
respect to the reference coordinates. 

For $F\in {\rm GL}^+(3)$, let
$
F=R \,U,
\ 
R=\operatorname{polar}F\in {\rm SO}(3),
\ 
U=\sqrt{F^TF}\in\operatorname{Sym}^{++}(3)
$
be the polar decomposition. Grioli's formula gives
\begin{align}
\operatorname{dist}^2(F,{\rm SO}(3))
=
|U-\id_3|^2;
\end{align}
see \cite{Grioli40,agn_neff2014grioli}. Although the Euclidean distance is not an intrinsic distance on ${\rm GL}^+(3)$ \cite{agn_neff2014riemannian,agn_neff2015geometry,agn_martin2014minimal}, the Biot strain $U-\id_3$ remains mechanically useful.

A representative first-gradient energy would be
\begin{align}
\varphi\longmapsto
\int_\Omega
\left[
\mu|U-\id_3|^2
+
\varepsilon|U-\id_3|^p
+
\frac{\lambda}{8}
\left(\det F-\frac1{\det F}\right)^2
\right]\,\mathrm dx,
\qquad p>\frac94. 
\end{align}
The determinant term controls volume collapse and large dilatations, but it does not in general restore the weak lower semicontinuity missing from the stretch term. Accordingly, we cannot claim existence for this unregularized first-gradient functional.

Instead, we introduce an independent microrotation $\overline  R :\Omega\to {\rm SO}(3)$ and the relative stretch
\begin{align}
\overline  U := \overline R ^T{\rm D}\varphi\in\mathbb R^{3\times3},
\end{align}
which need not be symmetric. For fixed $ \overline R $, the mapping ${\rm D}\varphi\mapsto \overline R ^T{\rm D}\varphi$ is linear. We supplement the stretch energy by a convex curvature energy depending on
\begin{align}
\alpha:= \overline R ^T\operatorname{Curl}  R .
\end{align}
Here the row-wise Curl is used. For rotation fields  in ${\rm W}^{1,q}(\Omega; {\rm SO}(3))$ the estimate \eqref{lem:rotation-curl-control} is satisfied.

The deformation-only model associated with the constrained lifting is the polar-curvature functional
\begin{align}
\label{eq:polar-prototype}
\mathcal J_{\rm Biot}(\varphi)
=
\int_\Omega
\Biggl[
&\mu|U-\id_3|^2
+
\varepsilon|U-\id_3|^p
+
\frac{\lambda}{8}
\left(\det F-\frac1{\det F}\right)^2
+{\rm L}_{\rm c}^q 
\left|R ^T\operatorname{Curl}\overline R \right|^q
\Biggr] \,\mathrm dx,
\end{align}
where $F={\rm D}\varphi$, $R=\operatorname{polar}F$, $p>9/4$, and $q\geq 2$. Its natural domain includes the requirements $\det{\rm D}\varphi>0$ and $R\in  {\rm W}^{1,q}(\Omega;{\rm SO}(3))$. For $F=\id_3+t\,{\rm D}u$ with $t\to0$, one has
$
U-\id_3
=
t\,\operatorname{sym}{\rm D}u+o(t),
$ $
\det F-\frac{1}{\det F}
=
2t\,\operatorname{tr}{\rm D}u+o(t).
$
Consequently, the quadratic part yields the classical linear elastic energy
$
\mu|\operatorname{sym}{\rm D}u|^2
+
\frac{\lambda}{2}(\operatorname{tr}{\rm D}u)^2.
$ The superquadratic term is of higher order near the identity and is introduced for global coercivity, as is customary in $(p,q)$-growth and related variational models \cite{Ogden72a,Ogden83,Marcellini1,Marcellini2}.

The curvature term is objective. For every constant $Q_0\in {\rm SO}(3)$,
\begin{align}
\operatorname{polar}(Q_0F)=Q_0 R,
\qquad
(Q_0 R)^T\operatorname{Curl}(Q_0 R)
=
 R^T\operatorname{Curl}  R.
\end{align}
 {The curvature term controls the variation of the rotational factor.
We recall that the compatibility relations derived in
\cite{LankeitNeffOsterbrink2017} allow, for sufficiently regular fields,
$\Curl \,U$ to be expressed algebraically in terms of $U$ and the rotational
curvature tensor. Hence, under suitable additional control of $U$ and
$U^{-1}$, the rotational curvature also yields corresponding information
on $\Curl U$.
This does not, however, amount to control of the full gradient ${\rm D}U$.
Indeed, from
$
{\rm D}\varphi=RU
$
one formally obtains
$
{\rm D}^2\varphi=({\rm D}R)U+R\,{\rm D}U,
$
and control of $\Curl \,U$ alone does not control all components of
${\rm D}U$. Consequently, the present rotational-curvature regularization
does not provide control of the full second gradient ${\rm D}^2\varphi$.
}
 Thus, it is weaker and more selective than a full second-gradient regularization. This is the essential distinction from models whose energy controls $\mathrm D^2\varphi$; compare, for example, the higher-order setting in \cite{vcevsik2025stability}.

The present framework is related to classical geometrically nonlinear Cosserat energies of the form
\begin{align}
\label{minprobc}
I_{\rm sym-skew}(\varphi, \overline R )
=
\int_\Omega\Biggl[
&\mu\left|\operatorname{sym}( \overline R ^T{\rm D}\varphi-\id_3)\right|^2
+
\mu_c\left|\operatorname{skew}( \overline R ^T{\rm D}\varphi-\id_3)\right|^2
\notag\\
&\qquad  +{\rm L}_{\rm c}^q \left| \overline R ^T\operatorname{Curl} \overline R \right|^q
+
\frac\lambda8
\left(\det{\rm D}\varphi-\frac1{\det{\rm D}\varphi}\right)^2
\Biggr] \,\mathrm dx;
\end{align}
see \cite{NeffSolid2004,neff2014existence,Neff_plate04_cmt,Neff_Chelminski_ifb07}. Formally sending $\mu_c\to\infty$ enforces symmetry of $ \overline R ^T{\rm D}\varphi$, but symmetry alone does not imply positive definiteness and therefore does not identify $\overline  R $ with the polar factor. The same obstruction remains if positivity is imposed only near or on the boundary. The following counterexample makes this point explicit. Since the constructed map is smooth in a neighborhood of $\partial\Omega$, the boundary statement below is a classical pointwise statement and does not rely on a Sobolev trace of ${\rm D}\varphi$.

\begin{counterexample}
[Interior symmetry, boundary positivity, and positive determinant a.e.\ do not imply interior positive definiteness]

Let
$
\Omega=B(0,1)\subset\mathbb R^3,
\ 
 \overline R (x)=\id_3,
$
and fix $a\in(0,1)$. Define
$
\psi(x)
=
\frac{1}{2}|x|^2-a|x|,
\ 
x\in B(0,1),
$
and set
$
\varphi(x)
=
{\rm D}\psi(x)
=
x-a\frac{x}{|x|}
$  $
\text{for }x\neq 0.
$
The value of $\varphi$ at $x=0$ can be chosen arbitrarily. Since
$
|{\rm D}\varphi(x)|
\leq
C\left(1+\frac{1}{|x|}\right),
$
we have
$
\varphi\in {\rm W}^{1,p}(B(0,1);\mathbb R^3)
$  $
\text{for every }1\leq p<3.$
In particular, one may choose any
$
p\in\left(\frac{9}{4},3\right).
$
For $x\neq 0$, let
$
n(x):=\frac{x}{|x|}.
$
Since $\varphi={\rm D}\psi$, its gradient is symmetric and is given by
\begin{align}
\begin{aligned}
{\rm D}\varphi(x)
=
{\rm D}^2\psi(x)
&=
\id_3
-
\frac{a}{|x|}
\left(
\id_3-n(x)\otimes n(x)
\right)
\\
&=
n(x)\otimes n(x)
+
\left(
1-\frac{a}{|x|}
\right)
\left(
\id_3-n(x)\otimes n(x)
\right).
\end{aligned}
\end{align}
\begin{figure}[t]
\centering

\begin{subfigure}[t]{0.32\textwidth}
\centering
\vspace{0pt}
\begin{tikzpicture}[scale=1.0]
    \def\R{2.25}
    \pgfmathsetmacro{\Ra}{\aval*\R}

    \fill[gray!18] (0,0) circle (\R);
    \fill[gray!38] (0,0) circle (\Ra);

    \draw[thick] (0,0) circle (\R);
    \draw[thick,dashed] (0,0) circle (\Ra);

    \fill (0,0) circle (1.2pt);

    \draw[thin]
      ({\Ra*cos(140)},{\Ra*sin(140)})
      -- ++(-0.45,0.35)
      node[above left,font=\scriptsize] {$r=a$};

    \draw[thin]
      ({\R*cos(220)},{\R*sin(220)})
      -- ++(-0.45,-0.20)
      node[below left,font=\scriptsize] {$r=1$};

    \node[font=\scriptsize, align=center] at (0,-2.85)
    {$0<r<a$: indefinite\\[1mm] $a<r<1$: positive definite};
\end{tikzpicture}
\caption{Geometry of the two regions.}
\label{fig:counterexample-geom}
\end{subfigure}
\hfill
\begin{subfigure}[t]{0.33\textwidth}
\centering
\vspace{0pt}

\begin{tikzpicture}
\begin{axis}[
    width=\linewidth,
    height=5.7cm,
    xmin=0.16,
    xmax=1.02,
    ymin=-2.7,
    ymax=1.35,
    axis lines=left,
    xlabel={$r=|x|$},
    ylabel={$\lambda$},
    xtick={\aval,1},
    xticklabels={$a$,$1$},
    ytick={-2,-1,0,1},
    grid=major,
    tick label style={font=\scriptsize},
    label style={font=\scriptsize},
    clip=false
]

\addplot[
    thick,
    domain=0.16:1,
    samples=2
]
{1};

\addplot[
    thick,
    dashed,
    domain=0.16:1,
    samples=300
]
{1-\aval/x};

\addplot[
    densely dotted,
    thick
]
coordinates {
    (\aval,-2.7)
    (\aval,1.25)
};

\node[
    font=\scriptsize,
    anchor=south west
]
at (axis cs:0.62,1.02)
{$\lambda_{\mathrm{rad}}=1$};

\node[
    font=\scriptsize,
    anchor=east
]
at (axis cs:1.1,-0.2)
{$\lambda_{\mathrm{tan}}=1-a/r$};

\end{axis}
\end{tikzpicture}

\caption{Radial and tangential eigenvalues.}
\label{fig:counterexample-eigs}
\end{subfigure}
\hfill
\begin{subfigure}[t]{0.32\textwidth}
\centering
\vspace{0pt}
\begin{tikzpicture}
\begin{axis}[
    width=\linewidth,
    height=5.7cm,
    xmin=0.16, xmax=1.02,
    ymin=0, ymax=6.2,
    axis lines=left,
    xlabel={$r=|x|$},
    ylabel={$\det{\rm D}\varphi$},
    xtick={\aval,1},
    xticklabels={$a$,$1$},
    ytick={0,2,4,6},
    grid=major,
    tick label style={font=\scriptsize},
    label style={font=\scriptsize}
]
    \addplot[thick,domain=0.16:1,samples=400] {(1-\aval/x)^2};
    \addplot[densely dotted,thick] coordinates {(\aval,0) (\aval,6.2)};
    \addplot[only marks,mark=*,mark size=1.5pt] coordinates {(\aval,0)};
\end{axis}
\end{tikzpicture}
\caption{Determinant of ${\rm D}\varphi$.}
\label{fig:counterexample-det}
\end{subfigure}

\caption{
Geometric and spectral structure of the counterexample.
The radial eigenvalue is $\lambda_{\rm rad}=1$, whereas the two tangential
eigenvalues are $\lambda_{\rm tan}=1-a/r$. Hence ${\rm D}\varphi$ is
indefinite for $0<r<a$, although
$\det{\rm D}\varphi=(1-a/r)^2>0$ for $r\neq a$.
}
\label{fig:counterexample-all}
\end{figure}
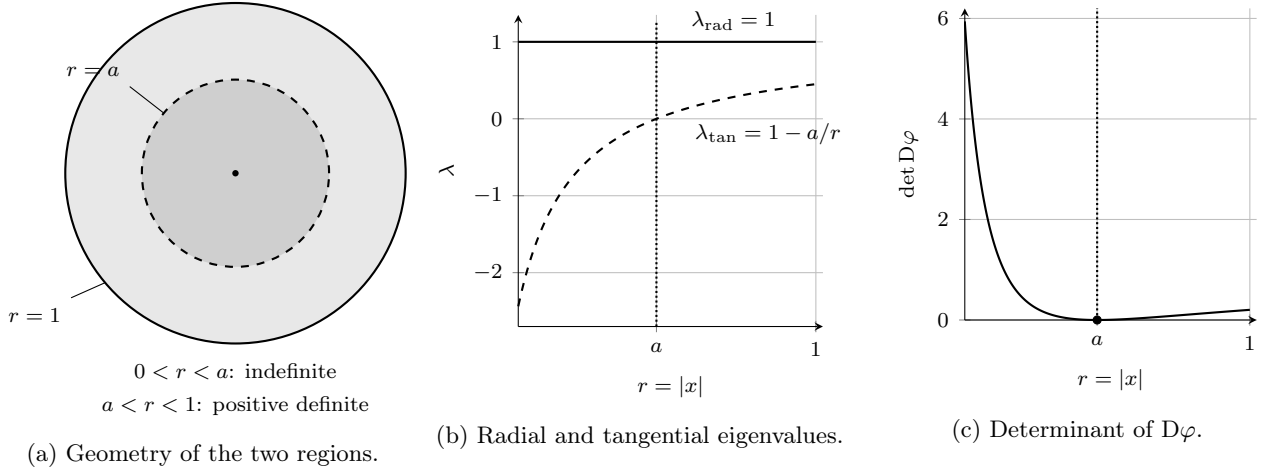
Hence
$
 \overline  R ^T{\rm D}\varphi
=
{\rm D}\varphi
\in\operatorname{Sym}(3)
\
\text{a.e. in }\Omega.
$
The eigenvalue of ${\rm D}\varphi(x)$ in the radial direction $n(x)$ is
$
\lambda_{\rm rad}(x)=1,
$
whereas the two tangential eigenvalues are
$
\lambda_{\rm tan}(x)
=
1-\frac{a}{|x|}.
$
Consequently,
\begin{align}
\det{\rm D}\varphi(x)
=
\left(
1-\frac{a}{|x|}
\right)^2
>0
\end{align}
for a.e.\ $x\in\Omega$. The determinant vanishes only on the sphere
$
\left\{
x\in\Omega~: |x|=a
\right\},
$
which has three-dimensional Lebesgue measure zero. Thus
$
\det{\rm D}\varphi>0
$ $
\text{a.e. in }\Omega.
$

On the boundary $\partial\Omega$, where $|x|=1$, the eigenvalues of
${\rm D}\varphi$ are
$
1,\  1-a,\  1-a.
$
Since $a\in(0,1)$, it follows that
$
 \overline R ^T{\rm D}\varphi
=
{\rm D}\varphi
\in\operatorname{Sym}^{++}(3)
$ $
\text{on }\partial\Omega.
$
On the other hand, for
$
0<|x|<a,
$
we have
$
1-\frac{a}{|x|}<0.
$
Thus ${\rm D}\varphi(x)$ has one positive eigenvalue and two negative
eigenvalues. In particular,
$
\overline  R ^T{\rm D}\varphi
=
{\rm D}\varphi
\notin\operatorname{Sym}^{++}(3)
$ $
\text{for }0<|x|<a.
$

Therefore, even the conditions
$
  \overline R ^T{\rm D}\varphi
\in\operatorname{Sym}(3)
 $ $
\text{a.e. in }\Omega,
$
$
\det{\rm D}\varphi>0
$  $
\text{a.e. in }\Omega,
$
and
$
  \overline R ^T{\rm D}\varphi
\in\operatorname{Sym}^{++}(3)
$ $
\text{on }\partial\Omega
$
do not imply
$
 \overline R ^T{\rm D}\varphi
\in\operatorname{Sym}^{++}(3)
$ $
\text{a.e. in }\Omega.
$
\end{counterexample}

The open cone $\operatorname{Sym}^{++}(3)$ is not weakly closed, so it cannot be imposed directly in a stable admissible class. We instead impose the closed-cone condition
$
 \overline R ^T{\rm D}\varphi\in\operatorname{Sym}^{+}(3)
$
almost everywhere and recover strict positivity from $\det{\rm D}\varphi>0$. Pointwise, if
$
S:= \overline R ^T{\rm D}\varphi\in\operatorname{Sym}^{+}(3),
$
then ${\rm D}\varphi= \overline R\,  S$ is a possibly singular polar-type factorization. If $\det{\rm D}\varphi>0$, then $S\in\operatorname{Sym}^{++}(3)$ and uniqueness of the polar decomposition gives
$
 \overline R =R:=\operatorname{polar}({\rm D}\varphi).
$
The weak continuity of the determinant and the lower semicontinuity of the singular volumetric density are therefore essential in the existence proof.

The prototype constrained energy associated with \eqref{eq:polar-prototype} is
\begin{align}
\label{eq:constrained-prototype}
I_{\rm Cosserat}(\varphi, \overline R )
=
\int_\Omega\Biggl[
&\mu| \overline R ^T{\rm D}\varphi-\id_3|^2
+
\varepsilon| \overline R ^T{\rm D}\varphi-\id_3|^p
+
{\rm L}_{\rm c}^q | \overline R ^T\operatorname{Curl} \overline R |^q
\notag\\
&\qquad \qquad \qquad \qquad \qquad \quad +
\frac\lambda8
\left(\det{\rm D}\varphi-\frac1{\det{\rm D}\varphi}\right)^2
\Biggr] \,\mathrm dx,
\end{align}
with $p>9/4$ and $q\geq 2$, minimized under
$
 \overline R ^T{\rm D}\varphi\in\operatorname{Sym}^{+}(3)
$
and $\det{\rm D}\varphi>0$ almost everywhere. 
The precise Cosserat--Biot combination considered here is not one of the standard forms recalled above, although each of its structural ingredients has clear precedents.  Ogden-type energies \cite{Ogden72a,Ogden83}, for instance, are written as sums
of powers of the principal stretches, with different exponents used to capture
different deformation regimes. On the  other hand, from the variational point of view, the simultaneous presence of different
growth exponents is also well established in the theory of functionals with
\((p,q)\)-growth and in double-phase problems \cite{Marcellini1,Marcellini2}. Moreover, the use of an additional superquadratic term should not be interpreted as a
modification of the linearized elastic moduli. The quadratic part determines
the tangent elasticity tensor at the identity, while the superquadratic term is
of higher order near \({\rm SO}(3)\) and is introduced for global coercivity.
This distinction between local quadratic behavior and global growth
assumptions is common in the variational theory of finite elasticity and coupled stress model. 
The relaxed Cosserat-Biot  model is related in spirit to higher-gradient approaches to
nonlinear elastodynamics, in particular to the framework of
{\v C}e{\v s}{\'i}k and Schwarzacher \cite{vcevsik2025stability}, where the time-discrete
minimization scheme is designed to handle elastic energies with nonlinearities
at the highest differential order. The distinction is, however, essential. In
our case the higher-order contribution is not a full second-gradient
regularization of the deformation; rather, it is a selective rotational higher-order regularization. We do not assume an energy depending on
\({\rm D}^2\varphi\), nor do we obtain a direct control of the full second
gradient of \(\varphi\). Instead, the higher-order term acts only on the
rotational part of the deformation gradient, namely on the Cosserat/polar field
$
     \overline R=R:=\operatorname{polar}({\rm D}\varphi),
$
through the curvature density
$
       |\overline R^T\Curl \overline R |^q .
$
Thus the model controls the spatial variation of the polar rotation, but not
the spatial variation of the stretch tensor. In this sense, the regularization
is geometrically constrained and weaker than a standard second-gradient
regularization, while still providing the compactness mechanism needed for the
constrained Cosserat-Biot formulation.

For the general theory we distinguish two admissible classes. The constrained class $\mathcal A_{\rm relax}$ was defined in the introduction. The auxiliary lifted Cosserat class is
\begin{align}
\mathcal A_{\rm Cosserat}
=
\Bigl\{
(\varphi, \overline R )&\in
{\rm W}^{1,p}(\Omega;\mathbb R^3)
\times  {\rm W}^{1,q}(\Omega;{\rm SO}(3)) \,|\,
\operatorname{Tr}_{\Gamma_d}\varphi=\varphi^*,
\ \det{\rm D}\varphi>0\ \text{a.e. in }\Omega
\Bigr\}.
\end{align}
On either class, with
$
 \overline U = \overline R ^T{\rm D}\varphi,
\ 
\alpha= \overline R ^T\operatorname{Curl} \overline R ,
$
we set
\begin{align}
\widehat{\mathcal I}(\varphi, \overline R )
:=
\int_\Omega W(x, \overline U ,\alpha)\,\mathrm dx,
\end{align}
where $W$ satisfies suitable conditions. We denote this functional by $\widehat{\mathcal I}$ on
$\mathcal A_{\rm Cosserat}$ and by $\mathcal I$ on
$\mathcal A_{\rm relax}$. 
If $\det{\rm D}\varphi>0$ and $\overline R=R:=\operatorname{polar}({\rm D}\varphi)$ belongs to $ {\rm W}^{1,q}(\Omega;{\rm SO}(3))$, then
\begin{align}
\widehat{\mathcal I}(\varphi,\overline R)
=
\mathcal I(\varphi, \overline R).
\end{align}
Conversely, every $(\varphi, \overline R )\in\mathcal A_{\rm relax}$ satisfies $ \overline R =R$ almost everywhere. This equivalence is formalized in Proposition~\ref{prop:polar-equivalence}.

\section{The interpretation of the coercivity properties}
\label{s2}\setcounter{equation}{0}

Let $\Omega\subset\mathbb R^3$ be bounded, connected and Lipschitz, and let $1<p<\infty$ and $q\geq 2$. This section explains why the two energy variables
$
\overline R^T{\rm D}\varphi-\id_3
\quad\text{and}\quad
\overline R^T\operatorname{Curl}\overline R
$
provide the required coercivity. Orthogonal invariance of the Frobenius norm gives
\begin{align}
|\overline R^T{\rm D}\varphi-\id_3|
=|{\rm D}\varphi-\overline  R|,
\qquad
|\overline R^T\operatorname{Curl}\overline R|
=|\operatorname{Curl}\overline R|.
\end{align}
Thus, the first quantity measures the mismatch between the deformation gradient and the microrotation, while the second controls the spatial variation of the rotation field.

The estimate for sufficiently regular rotation fields is precisely the pointwise Curl--Grad inequality of Neff and M\"unch \cite{Neff_curl06}. For$q\geq 2$ and every $\overline R\in {\rm W}^{1,q}(\Omega;{\rm SO}(3)),
$ one has
\begin{align}
\qquad
\|{\rm D} \overline R\|_{{\rm L}^q(\Omega)}
\leq
C\|\operatorname{Curl}\overline R\|_{{\rm L}^q(\Omega)},
\end{align}
where $C$ depends only on the dimension.

\begin{proposition}[Coercivity modulo rigid motions]
\label{prop:lifted-cosserat-coercivity}
For every pair satisfying
$
\varphi\in {\rm W}^{1,p}(\Omega;\mathbb R^3),
$ $
\overline R\in  {\rm W}^{1,q}(\Omega;{\rm SO}(3)) ,
$
where $1<p<\infty$, $q\geq 2$, 
there exists a constant rotation $ Q \in {\rm SO}(3)$ such that
\begin{align}
\label{eq:lifted-cosserat-coercivity}
&\|{\rm D}\varphi- Q \|_{{\rm L}^p(\Omega)}
+
\|\overline R- Q \|_{{\rm W}^{1,q}(\Omega)}
\leq
C\left(
\|\overline R^T{\rm D}\varphi-\id_3\|_{{\rm L}^p(\Omega)}
+
\|\overline R^T\operatorname{Curl}\overline R\|_{{\rm L}^q(\Omega)}
\right),
\end{align}
where $C$ depends only on $\Omega,p,q$. No orientation condition on ${\rm D}\varphi$ is required.
\end{proposition}

\begin{proof}
The estimate~\ref{lem:rotation-curl-control} yields
\begin{align}
\label{eq:curl-grad-control}
\|{\rm D} \overline R\|_{{\rm L}^q(\Omega)}
\leq
C\|\overline R^T\operatorname{Curl} \overline R\|_{{\rm L}^q(\Omega)}.
\end{align}
Assume first that $q\leq p$. Since
$
\operatorname{dist}({\rm D}\varphi,{\rm SO}(3))
\leq|{\rm D}\varphi-\overline R|,
$
the ${\rm L}^p$ geometric rigidity estimate of Friesecke--James--M\"uller \cite{Mueller02} and Conti--Dolzmann--M\"uller \cite{conti2014korn} provides $ Q \in {\rm SO}(3)$ such that
\begin{align}
\|{\rm D}\varphi- Q \|_{{\rm L}^p}
\leq
C\|\overline R^T{\rm D}\varphi-\id_3\|_{{\rm L}^p}.
\end{align}
The triangle inequality gives the same ${\rm L}^p$-control for $\overline R- Q$. Since $q\leq p$ and $\Omega$ is bounded, its ${\rm L}^q$ norm is also controlled. Combining this with \eqref{eq:curl-grad-control} proves \eqref{eq:lifted-cosserat-coercivity}.

Assume now that $q>p$. Define the mean
$
Z :=|\Omega|^{-1}\int_\Omega \overline R\,\mathrm dx\notin{\rm SO}(3) \ \text{(in general)},
$
and choose $ Q \in {\rm SO}(3)$ such that
$
|Z - Q |=\operatorname{dist}(Z ,{\rm SO}(3)).
$
For each $x$, $\operatorname{dist}(Z ,{\rm SO}(3))\leq|Z -\overline R(x)|$. Hence Poincar\'e's inequality and \eqref{eq:curl-grad-control} imply
\begin{align}
\|\overline R- Q \|_{{\rm W}^{1,q}(\Omega)}
\leq
C\|\overline R^T\operatorname{Curl}\overline R\|_{{\rm L}^q(\Omega)}.
\end{align}
The embedding ${\rm L}^q(\Omega)\hookrightarrow {\rm L}^p(\Omega)$ then controls $\overline R- Q $ in ${\rm L}^p$. Finally,
\begin{align}
\|{\rm D}\varphi-Q\|_{{\rm L}^p}
&\leq
\|{\rm D}\varphi-\overline R\|_{{\rm L}^p}
+
\|\overline R-Q\|_{{\rm L}^p}
=
\|\overline R^T{\rm D}\varphi-\id_3\|_{{\rm L}^p}
+
\|\overline R-Q\|_{{\rm L}^p}
\\
&\leq
C\left(
\|\overline R^T{\rm D}\varphi-\id_3\|_{{\rm L}^p}
+
\|\overline R^T\operatorname{Curl}\overline R\|_{{\rm L}^q}
\right).\qedhere\notag
\end{align}

\end{proof}

The two terms in the energy have complementary roles. If both defects vanish, then $\overline R$ is a constant rotation and $\varphi(x)= \overline R\ x+c$. In the polar specialization, if $\det{\rm D}\varphi>0$ and
$
\overline R=R=\operatorname{polar}({\rm D}\varphi) 
\in {\rm W}^{1,q}(\Omega;{\rm SO}(3)),
$
then
$
\overline R^T{\rm D}\varphi-\id_3
=U-\id_3.
$
Proposition~\ref{prop:lifted-cosserat-coercivity} therefore controls the polar rotation in ${\rm W}^{1,q}$ and the deformation gradient modulo a constant rotation. It does not control $\mathrm D U$ and certainly does not provide ${\rm W}^{1,q}$-control of the full deformation gradient.

\section{Existence for a Cosserat model with determinant blow-up}
\label{s3}\setcounter{equation}{0}

We first minimize over the auxiliary class $\mathcal A_{\rm Cosserat}$, without the positive-semidefinite constraint on the relative stretch.

\begin{theorem}[Existence for the auxiliary lifted problem]
\label{thm:elastic2}
Assume that $\mathcal A_{\rm Cosserat}$ contains a pair of finite energy. If either H1)--H4) or H1')--H4') holds, then
$
\inf_{(\varphi, \overline R )\in\mathcal A_{\rm Cosserat}}
\widehat{\mathcal I}(\varphi, \overline R )
$
admits a minimizer.
\end{theorem}

\begin{proof}
Let $(\varphi_k, \overline R _k)\subset\mathcal A_{\rm Cosserat}$ be a minimizing sequence with uniformly bounded energy and set
\begin{align}
\overline U_k&:= \overline R _k^T{\rm D}\varphi_k,
&J_k&:=\det \overline U_k=\det{\rm D}\varphi_k,
\notag\\
\Gamma_k&:=\operatorname{Cof}\overline U_k
= \overline R _k^T\operatorname{Cof}{\rm D}\varphi_k,
&\alpha_k&:= \overline R _k^T\operatorname{Curl} \overline R _k.
\end{align}
The identities follow from $\det \overline R _k=1$ and
$
\operatorname{Cof}( \overline R _k^TF)
= \overline R _k^T\operatorname{Cof}F.
$

In both regimes, coercivity bounds $(\alpha_k)$ in ${\rm L}^q$. Since
$
|\operatorname{Curl} \overline R _k|=|\alpha_k|,
$
the estimate~\ref{lem:rotation-curl-control} gives a uniform ${\rm W}^{1,q}$ bound for $ \overline R _k$. Because $ \overline R _k$ is pointwise bounded, after extraction there exists $ \overline R \in {\rm W}^{1,q}(\Omega;\mathbb R^{3\times3})$ such that
\begin{align}
 \overline R _k&\rightharpoonup \overline R 
&&\text{in }{\rm W}^{1,q}(\Omega;\mathbb R^{3\times3}),
\notag\\
 \overline R _k&\to \overline R 
&&\text{strongly in }{\rm L}^t(\Omega;\mathbb R^{3\times3})
\quad\text{for every }1\leq t<\infty,
\end{align}
and almost everywhere. Since ${\rm SO}(3)$ is closed, $ \overline R (x)\in {\rm SO}(3)$ almost everywhere.
Moreover,
\begin{align}
\alpha_k
\rightharpoonup
 \overline R ^T\operatorname{Curl} \overline R 
=: \alpha
\quad\text{in }{\rm L}^q(\Omega;\mathbb R^{3\times3}).
\end{align}
Indeed\footnote{If \(m>1\), we denote by $m'=\frac{m}{m-1}$ its H\"older conjugate exponent.}, for every $\Theta\in {\rm L}^{q'}$,
\begin{align}
\int_\Omega\langle\alpha_k,\Theta\rangle\,\mathrm dx
=
\int_\Omega
\langle\operatorname{Curl} \overline R _k, \overline R _k\Theta\rangle\,\mathrm dx,
\end{align}
and $ \overline R _k\Theta\to \overline R \Theta$ strongly in ${\rm L}^{q'}$ by dominated convergence.

\medskip
\noindent\emph{Cofactor-dependent regime.}
By H4), the sequences $({\rm D}\varphi_k)$, $(\operatorname{Cof}{\rm D}\varphi_k)$, $(\det{\rm D}\varphi_k)$ and $(\alpha_k)$ are bounded in ${\rm L}^p,{\rm L}^r,{\rm L}^s,{\rm L}^q$, respectively. The trace condition and Poincar\'e's inequality give a uniform ${\rm W}^{1,p}$ bound for $\varphi_k$. Therefore, along a subsequence,
\begin{align}
\varphi_k&\rightharpoonup\varphi
&&\text{in }{\rm W}^{1,p}(\Omega;\mathbb R^3),
\notag\\
\operatorname{Cof}{\rm D}\varphi_k&\rightharpoonup H
&&\text{in }{\rm L}^r(\Omega;\mathbb R^{3\times3}),
\notag\\
\det{\rm D}\varphi_k&\rightharpoonup J
&&\text{in }{\rm L}^s(\Omega).
\end{align}
Since $p\geq2$, the second-order minors converge distributionally by the
weak continuity of null Lagrangians
\cite{Ball1977,BallCurrieOlver1981}.
Together with the ${\rm L}^r$ bound, $r>1$, this identifies
$
H=\operatorname{Cof}{\rm D}\varphi.
$

If $p<3$, the strict condition
$
\frac1p+\frac1r<\frac43
$
is equivalent to
$
r'<p^*:=\frac{3p}{3-p}.
$
Hence, by the Rellich--Kondrachov theorem \cite{AdamsFournier2003},
\begin{equation}
\varphi_k\to\varphi
\qquad\text{strongly in }L^{r'}(\Omega;\mathbb R^3).
\end{equation}
For $p=3$, the same conclusion holds for every finite exponent $r'$. If $p>3$, the compact Morrey embedding yields \cite{AdamsFournier2003}, for every
$
0\leq\theta<1-\frac3p,
$
that, up to a subsequence,
\begin{equation}
\varphi_k\to\varphi
\qquad\text{strongly in }
C^{0,\theta}(\overline\Omega;\mathbb R^3),
\end{equation}
and therefore strongly in $L^{r'}(\Omega;\mathbb R^3)$. We may consequently pass to the limit in the weak--strong product appearing in the Piola representation of the determinant.
Using the Piola identity,
\begin{align}
\det{\rm D}\varphi_k
=
\frac13\operatorname{div}\left[
(\operatorname{Cof}{\rm D}\varphi_k)^T\varphi_k
\right]
\quad\text{in }\mathcal D'(\Omega),
\end{align}
we may pass to the limit in the weak--strong product and obtain
\begin{align}
\det{\rm D}\varphi_k
\longrightarrow
\det{\rm D}\varphi
\quad\text{in }\mathcal D'(\Omega).
\end{align}
The weak ${\rm L}^s$ limit is unique, so $J=\det{\rm D}\varphi$ almost everywhere.

The strong convergence of $ \overline R _k$ and the weak convergence of the classical minors now yield
\begin{align}
\overline U_k&\rightharpoonup \overline U:= \overline R ^T{\rm D}\varphi
&&\text{in }{\rm L}^p,
\notag\\
\Gamma_k&\rightharpoonup\operatorname{Cof}\overline U
&&\text{in }{\rm L}^r,
\notag\\
J_k&\rightharpoonup\det \overline U
&&\text{in }{\rm L}^s.
\end{align}
For example, for $\Psi\in {\rm L}^{p'}$,
\begin{align}
\int_\Omega\langle \overline U_k,\Psi\rangle\,\mathrm dx
=
\int_\Omega\langle{\rm D}\varphi_k, \overline R _k\Psi\rangle\,\mathrm dx
\longrightarrow
\int_\Omega\langle{\rm D}\varphi, \overline R \Psi\rangle\,\mathrm dx,
\end{align}
and the cofactor convergence is analogous.

By H2)--H3), the extended density is a convex normal integrand. Weak lower semicontinuity therefore gives
\begin{align}
\widehat{\mathcal I}(\varphi, \overline R )
\leq
\liminf_{k\to\infty}
\widehat{\mathcal I}(\varphi_k, \overline R _k).
\end{align}
Its finiteness and the value $+\infty$ for non-positive determinants imply
$
\det{\rm D}\varphi=\det \overline U>0
$
almost everywhere. The trace condition passes to the limit, and hence $(\varphi, \overline R )\in\mathcal A_{\rm Cosserat}$.

\medskip
\noindent\emph{Reduced determinant-dependent regime.}
By H4'), $({\rm D}\varphi_k)$, $(\det{\rm D}\varphi_k)$ and $(\alpha_k)$ are bounded in ${\rm L}^p,{\rm L}^s,{\rm L}^q$, where $p>9/4$ and $s>1$, $q\geq 2$. We again obtain
\begin{align}
\varphi_k\rightharpoonup\varphi
\quad\text{in }{\rm W}^{1,p},
\qquad
\det{\rm D}\varphi_k\rightharpoonup J
\quad\text{in }{\rm L}^s.
\end{align}
Since $p>2$, the cofactors are bounded in the reflexive space ${\rm L}^{p/2}$ and
\begin{align}
\operatorname{Cof}{\rm D}\varphi_k
\rightharpoonup
\operatorname{Cof}{\rm D}\varphi
\quad\text{in }{\rm L}^{p/2}(\Omega;\mathbb R^{3\times3}).
\end{align}

For $2<p<3$, choose $t$ such that
$
\frac{p}{p-2}\leq t<\frac{3p}{3-p}.
$
Such a $t$ exists precisely when
$
\frac{p}{p-2}<\frac{3p}{3-p}
$, i.e., when $p>\frac94.
$
The compact embedding gives $\varphi_k\to\varphi$ strongly in ${\rm L}^t$, and
$
\frac 2p+\frac1t\leq1.
$
The Piola representation therefore yields
\begin{align}
\det{\rm D}\varphi_k
\longrightarrow
\det{\rm D}\varphi
\quad\text{in }\mathcal D'(\Omega).
\end{align}
For $p=3$, one may choose any finite $t$, and for $p>3$ the conclusion follows from Sobolev--Morrey compactness \cite{AdamsFournier2003}. By uniqueness of the distributional and weak ${\rm L}^s$ limits,
$
J=\det{\rm D}\varphi.
$
This proves the weak continuity of the Jacobian needed in the reduced regime without an independent cofactor estimate.

As before,
\begin{align}
\overline U_k\rightharpoonup \overline U:= \overline R ^T{\rm D}\varphi
\quad\text{in }{\rm L}^p,
\qquad
J_k\rightharpoonup\det \overline U
\quad\text{in }{\rm L}^s,
\end{align}
and $\alpha_k\rightharpoonup\alpha$ in ${\rm L}^q$. The convex-normal-integrand hypothesis H2')--H3') gives weak lower semicontinuity, while finite limiting energy gives $\det{\rm D}\varphi>0$ almost everywhere. Thus the limit belongs to $\mathcal A_{\rm Cosserat}$ and minimizes $\widehat{\mathcal I}$.
\end{proof}

\section{Existence result for the relaxed nonlinear elastic model}
\label{s4}\setcounter{equation}{0}

We now impose the weakly closed positive-semidefinite constraint and identify the independent rotation with the polar factor.

\begin{proposition}[Polar recovery and equivalence of admissible classes]
\label{prop:polar-equivalence}
Define
\begin{align}
\mathcal A_{\rm pol}
:=
\Bigl\{
\varphi\in {\rm W}^{1,p}(\Omega;\mathbb R^3)\,|\,
&\operatorname{Tr}_{\Gamma_d}\varphi=\varphi^*,
\quad
\det{\rm D}\varphi>0\ \text{a.e.},
\notag\\[-1mm]
& R:=\operatorname{polar}({\rm D}\varphi)
\in  {\rm W}^{1,q}(\Omega;{\rm SO}(3)) 
\Bigr\}.
\end{align}
Then
$
\varphi\mapsto(\varphi,R)
$
is a bijection from $\mathcal A_{\rm pol}$ onto $\mathcal A_{\rm relax}$, and, therefore, the admissible classes $\mathcal A_{\rm pol}$ onto $\mathcal A_{\rm relax}$ are equivalent. For every corresponding pair,
$
R^T{\rm D}\varphi
=
\sqrt{({\rm D}\varphi)^T{\rm D}\varphi}
\in\operatorname{Sym}^{++}(3)
\text{a.e.}
$
Moreover, if
\begin{align}
\mathcal J(\varphi)
:=
\int_\Omega
W\left(x,
\sqrt{({\rm D}\varphi)^T{\rm D}\varphi},
R^T\operatorname{Curl}\overline R
\right)\,\mathrm dx,
\end{align}
then
$
\mathcal J(\varphi)=\mathcal I(\varphi,R).
$
\end{proposition}

\begin{proof}
If $\varphi\in\mathcal A_{\rm pol}$, the polar decomposition gives
$
R^T{\rm D}\varphi
=
\sqrt{({\rm D}\varphi)^T{\rm D}\varphi}
\in\operatorname{Sym}^{++}(3),
$
so $(\varphi,R)\in\mathcal A_{\rm relax}$.

Conversely, let $(\varphi, \overline R )\in\mathcal A_{\rm relax}$ and set
$
S:= \overline R ^T{\rm D}\varphi.
$
Then $S\in\operatorname{Sym}^{+}(3)$ and
$
\det S=\det{\rm D}\varphi>0.
$
Hence $S\in\operatorname{Sym}^{++}(3)$, and
$
{\rm D}\varphi= \overline R\,  S
$
is the unique polar decomposition. Therefore $ \overline R =R:=\operatorname{polar}({\rm D}\varphi)$ and $S=\sqrt{({\rm D}\varphi)^T{\rm D}\varphi}$. The equality of the functionals follows by substitution.
\end{proof}

\begin{theorem}[Existence in the constrained and polar formulations]
\label{thm:elastic3}
Assume that $\mathcal A_{\rm relax}$ contains a pair of finite energy. Under either H1)--H4) or H1')--H4'), the functional $\mathcal I$ admits a minimizer in $\mathcal A_{\rm relax}$. Equivalently, $\mathcal J$ admits a minimizer in $\mathcal A_{\rm pol}$.
\end{theorem}

\begin{proof}
Let $(\varphi_k, \overline R _k)\subset\mathcal A_{\rm relax}$ be a minimizing sequence. All compactness, identification of minors and lower-semicontinuity arguments from Theorem~\ref{thm:elastic2} apply without change. Thus, along a subsequence,
\begin{align}
\varphi_k&\rightharpoonup\varphi
&&\text{in }{\rm W}^{1,p}(\Omega;\mathbb R^3),
\notag\\
 \overline R _k&\to \overline  R 
&&\text{a.e. and strongly in every finite }{\rm L}^a,
\notag\\
\overline U_k:= \overline R _k^T{\rm D}\varphi_k
&\rightharpoonup
\overline U:= \overline R ^T{\rm D}\varphi
&&\text{in }{\rm L}^p(\Omega;\mathbb R^{3\times3}),
\end{align}
and the limit has positive determinant and finite energy.

For every $k$,
$
\overline U_k(x)\in\operatorname{Sym}^{+}(3)
$
almost everywhere. The set
\begin{align}
\left\{
V\in {\rm L}^p(\Omega;\mathbb R^{3\times3})\, |\,
V(x)\in\operatorname{Sym}^{+}(3)\ \text{a.e.}
\right\}
\end{align}
is convex and strongly closed, hence weakly closed. Therefore
$
\overline U= \overline R ^T{\rm D}\varphi\in\operatorname{Sym}^{+}(3)
$
almost everywhere. The trace condition also passes to the limit, so $(\varphi, \overline R )\in\mathcal A_{\rm relax}$. Weak lower semicontinuity gives minimality.

Proposition~\ref{prop:polar-equivalence} identifies $ \overline R $ with $\operatorname{polar}({\rm D}\varphi)$ and transfers the minimizer bijectively to $\mathcal A_{\rm pol}$.
\end{proof}

The theorem establishes local orientation preservation through $\det{\rm D}\varphi>0$ almost everywhere. Global injectivity would require an additional condition, such as a suitable
Ciarlet--Ne\v{c}as-type constraint \cite{CiarletNecas1987}, and is not
addressed here.

\section{Representative energy classes}\label{s6}\setcounter{equation}{0}

All examples are understood with the value $+\infty$ for non-positive determinants and must satisfy the measurability and lower-semicontinuity requirements H3) or H3'). The displayed growth conditions are sufficient and are not claimed to be optimal. To keep the main argument focused, we present representative families: some reduced determinant-dependent models and some genuinely cofactor-dependent models. Further admissible examples, together with models that fall outside the hypotheses without additional modification, are collected in Appendix~\ref{app:further-examples}.

Throughout this section, the volumetric function
$
h:(0,\infty)\to\mathbb R
$
is assumed to be convex and lower semicontinuous. Moreover, there exist
$c_h,\gamma_h>0$, $C_h\geq0$, $s>1$, and $\beta>0$ such that
\[
h(d)
\geq
c_h \,d^s-C_h
\qquad\text{for all }d>0.
\]
The function is extended by the value $+\infty$ for $d\leq0$, $\lim _{r\to 0} h(d)=+\infty$ such that the extension is convex.

For each of the stretching energies listed below, the curvature contribution may be chosen independently in the canonical form
\begin{align}
W_{\rm curv}(x,K)
=
a(x)|K|^q,
\qquad q\geq 2,
\end{align}
where $a:\Omega\to\mathbb R$ is measurable and satisfies
$
0<a_0\leq a(x)\leq a_1<+\infty
\ \text{for almost every }x\in\Omega.
$
In particular, one may take
\begin{align}
W_{\rm curv}(K)
=
{\rm L}_{\rm c}^{q}|K|^q,
\qquad {\rm L}_{\rm c}>0.
\end{align}
Since the mapping
$
K\longmapsto |K|^q
$
is convex and continuous for $q\geq1$, and strictly convex for $q\geq 2$, this choice satisfies the convexity, lower-semicontinuity and $q$-coercivity assumptions imposed on the curvature energy in the existence theorems. Moreover,
\begin{align}
\bigl| \overline R ^T\Curl \overline R \bigr|
=
\bigl|\Curl \overline R \bigr|,
\end{align}
because $ \overline R \in{\rm SO}(3)$. Hence, this curvature term provides precisely the $L^q$-control of $\Curl \overline R $ required for the compactness of the microrotation field. Consequently, for the examples below it remains only to specify suitable $W_{\rm stretch}$ energies.

\subsection{Reduced determinant-dependent models}

The following energies have the form
\begin{align}
        W(x,\overline U, \overline R ^T\Curl \overline R )
        =
        W_{\rm stretch}\bigl(x,\overline U,\det \overline U\bigr)
        +
        W_{\rm curv}\bigl(x, \overline R ^T\Curl \overline R \bigr).
\end{align}
They do not depend on $\Cof \overline U$ as an independent variable. Therefore they fit
the reduced hypotheses $H1')-H4')$, provided
$
        p>\frac{9}{4},
           $ $
        s>1,
           $ $
        q\geq 2.
$

\begin{itemize}

\item \textbf{Regularized Biot energy.}
The Biot energy naturally depends on the stretch itself, not on the cofactor.
An admissible regularized density is
\begin{align}
        W_{\rm stretch}\bigl(x,\overline U,\det \overline U\bigr)
        =
        \mu  | \overline U-\id_3 | ^2
        +
        \varepsilon | \overline U-\id_3 | ^p
        +
        h(\det \overline U),
\end{align}
where
$
        \mu\geq0,
             $ $
        \varepsilon>0,
            $ $
        p>\frac{9}{4}. 
$
The quadratic term is the classical Biot contribution, while the $p$-term gives
the coercivity required in $H4')$.

\item \textbf{Compressible regularized Neo-Hookean decoupled type energy.}
The compressible Neo-Hookean energy is naturally of reduced type: it depends on
the first invariant and on the determinant, but not independently on the
cofactor. In the present setting one may take
\begin{align}
        W_{\rm stretch}\bigl(x,\overline U,\det \overline U\bigr)
        =
        \mu |\overline U|^p
        +
        h(\det \overline U),
        \qquad
        \mu>0,
        \qquad
        p>\frac{9}{4}.
\end{align}
We call this energy \emph{decoupled} since the stretch contribution does not
depend on the independent microrotation $\overline R$. Indeed,
\begin{align}
        W_{\rm stretch}\bigl(x,\overline U,\det \overline U\bigr)
        &=
        \mu |\overline R^T{\rm D}\varphi|^p
        +
        h\bigl(\det(\overline R^T{\rm D}\varphi)\bigr)
        =
        \mu |{\rm D}\varphi|^p
        +
        h(\det{\rm D}\varphi),
\end{align}
and {\it the coupling is only in the admissible set}.
Thus, the stretch contribution is already polyconvex in the classical
deformation gradient. However, for $9/4<p<3$ it does not provide an
independent coercive bound on $\Cof{\rm D}\varphi$ of the type entering the
standard Ball existence theory. The reduced existence result proved here
shows that such an independent cofactor coercivity is not needed: the
condition $p>9/4$, together with the determinant control, is sufficient to
identify the weak Jacobian limit. After imposing the polar constraint, this
also yields existence for the corresponding energy supplemented by a
curvature term depending on
$R=\operatorname{polar}({\rm D}\varphi)$.

The purely quadratic choice $p=2$ corresponds to the classical first-invariant
part, but it does not satisfy $H4')$, which requires $p>\frac{9}{4}$. Therefore, if one
wants to keep the classical quadratic term, one should use the regularized form
\begin{align}
        W_{\rm stretch}\bigl(x,\overline U,\det \overline U\bigr)
        =
        \mu  | \overline U | ^2
        +
        \varepsilon  | \overline U | ^p
        +
        h(\det \overline U),
        \qquad
        \mu\geq0,
        \qquad
        \varepsilon>0,
        \qquad
        p>\frac{9}{4}.
\end{align}

\item \textbf{Anisotropic convex stretch energies.}
Let $\mathbb C(x)$ be a measurable uniformly positive definite fourth-order
tensor on all of $\mathbb{R}^{3\times 3}$. Then
\begin{align}
        W_{\rm stretch}\bigl(x,\overline U,\det \overline U\bigr)
        =
        \left\langle
        \mathbb C(x)(\overline U-\id_3),\overline U-\id_3
        \right\rangle^{p/2}
        +
        h(\det \overline U)
\end{align}
satisfies $H1')-H4')$ for
$
        p>\frac{9}{4}.
$

\end{itemize}

The regularized Biot model displays the principal stretch-based application of the reduced theorem, whereas the regularized Neo-Hookean model shows that the result is not tied to a shifted strain measure.

\subsection{Cofactor-dependent models}

\begin{itemize}
\item \textbf{Cofactor-dependent Biot-type coupled energy.}
An example for which the independent microrotation cannot
be eliminated from the stretch contribution is
\begin{align}
W_{\rm stretch}
\bigl(x,\overline U,\Cof\overline U,\det\overline U\bigr)
={}&
\mu|\overline U-\id_3|^p
+
b|\Cof\overline U-\id_3|^r
+
h(\det\overline U),
\end{align}
where
$
\mu,b>0, $ $
p\geq2,$ $
r>1,$  $
\frac1p+\frac1r<\frac43,
$
and $h:(0,\infty)\to\mathbb R$ is convex, with the corresponding
orientation-preserving extension by $+\infty$ for non-positive
determinants and with the growth required in H4).

Indeed,
\begin{align}
|\overline U-\id_3|
&=
|{\rm D}\varphi-\overline R|,
\qquad 
|\Cof\overline U-\id_3|
=
|\Cof{\rm D}\varphi-\overline R|,
\end{align}
so that, in contrast to the decoupled terms
$|\overline U|^p$ and $|\Cof\overline U|^r$, the microrotation
$\overline R$ cannot be removed from the energy.

After imposing
$\overline R^T{\rm D}\varphi\in\operatorname{Sym}^+(3)$ and
$\det{\rm D}\varphi>0$, one has
$\overline R=R=\operatorname{polar}({\rm D}\varphi)$ and therefore
$\overline U=U=\sqrt{({\rm D}\varphi)^T{\rm D}\varphi}$. Consequently,
the corresponding deformation-only stretch energy is
\begin{align}
W_{\rm stretch}
=
\mu|U-\id_3|^p
+
b|\Cof U-\id_3|^r
+
h(\det U),
\end{align}
supplemented in the full model by the curvature energy of the polar factor.
This provides a cofactor-dependent analogue of the regularized Biot model
which is genuinely covered by the Cosserat lifting mechanism rather than
reducing to a classical polyconvex energy in ${\rm D}\varphi$.

\item \textbf{Mooney-Rivlin  decoupled  type energy.}
Mooney-Rivlin energies are the natural examples where the cofactor variable
appears. A lifted Cosserat-polyconvex version is
\begin{align}
        W_{\rm stretch}\bigl(x,\overline U,\Cof \overline U,\det \overline U\bigr)
        =
        a | \overline U | ^p
        +
        b | \Cof \overline U | ^r
        +
        h(\det \overline U),
\end{align}
where
$
        a,b>0,
              $ $
        p\geq2,
         $ $
        r>1,
         $ $
        \frac1p+\frac1r<\frac43.
$
The classical quadratic choice
$
        p=r=2
$
is admissible.

 {
This energy is decoupled at the level of the stretch density, since its
dependence on $\overline R$ disappears. For $p=r=2$ it reduces to the
classical quadratic Mooney--Rivlin-type polyconvex form, while general
$p$ and $r$ give a corresponding power-growth extension. The deformation
and microrotation nevertheless remain coupled through the admissible set.
If the constraint coupling ${\rm D}\varphi$ and $\overline R$ were removed,
the corresponding Cosserat minimization problem would become completely
decoupled, a situation which is not considered here.}

 {This and the following decoupled examples are included primarily to show
that the present framework contains, and in the presence of the
rotational-curvature term extends, the classical existence theory of Ball
for isotropic nonlinear elasticity.}

The Ciarlet-Geymonat class \cite{CiarletGeymonat1982} is naturally cofactor-dependent. In the present
variables, a representative density is
\begin{align}
        W_{\rm stretch}\bigl(x,\overline U,\Cof \overline U,\det \overline U\bigr)
        =
        a | \overline U | ^p
        +
        b | \Cof \overline U | ^r
        +
        c(\det \overline U)^s
        -
        \gamma\log\det \overline U,
\end{align}
where
$
        a, b, c, \gamma>0,
            $ $
        p\geq2,
             $ $
        r>1,
             $ $
        s>1,
           $ $
        \frac1p+\frac1r<\frac43.
$
The term $-\log\det \overline U$ gives singular behavior as $\det \overline U\to0^+$, while
the powers provide the coercivity required by $H4)$.

One may also use the normalized Ciarlet-Geymonat form
\begin{align}
\begin{aligned}
        W_{\rm stretch}\bigl(x,\overline U,\Cof \overline U,\det \overline U\bigr)
        ={}&
        \frac{\mu}{2}
        \left(
         | \overline U | ^2
        -
        2\log\det \overline U
        -
        3
        \right)
        \\
        &+
        b | \Cof \overline U | ^r
        +
        \frac{\lambda}{4}
        \left(
        (\det \overline U)^2
        -
        2\log\det \overline U
        -
        1
        \right)
        \\
        &\quad+
        \varepsilon |\overline U|^p
        +
        c(\det \overline U)^s,
\end{aligned}
\end{align}
with
$
\mu>0,\ \lambda>0,\  b,\varepsilon,c>0,
\ p\geq2,\  r>1,\  s>1,
\ \frac1p+\frac1r<\frac43 .
$
If the quadratic term $|\overline U|^2$ is retained as the leading $U$-growth, one may
take $p=2$, provided the cofactor exponent satisfies
$
        r>\frac65.
$

\item \textbf{Anisotropic stretch-cofactor energy.}
Let $\mathbb C(x)$ and $\mathbb D(x)$ be measurable uniformly positive
definite fourth-order tensors on whole matrix space  $\mathbb R^{3\times 3}$. Then
\begin{align}
\begin{aligned}
        W_{\rm stretch}\bigl(x,\overline U,\Cof \overline U,\det \overline U\bigr)
        ={}&
        \left\langle
        \mathbb C(x)(\overline U-\id_3),\overline U-\id_3
        \right\rangle^{p/2}
       +
        \left\langle
        \mathbb D(x)(\Cof \overline U-\id_3),\Cof \overline U-\id_3
        \right\rangle^{r/2}\notag
        \\&+
        h(\det \overline U)
\end{aligned}
\end{align}
satisfies $H1)-H4)$ if
$
        p\geq2,
             $ $
        r>1,
           $ $
        \frac1p+\frac1r<\frac43.
$

\end{itemize}

These  examples already show the two distinct mechanisms covered by the theory: determinant-dependent stretch energies without an independent cofactor estimate, and classical polyconvex-type energies with genuine cofactor growth.

\subsection{Models not covered without additional modification}

\begin{itemize}

\item The purely quadratic reduced Neo-Hookean energy
\begin{align}
        W_{\rm stretch}\bigl(x,\overline U,\det \overline U\bigr)
        =
        \mu  | \overline U | ^2+h(\det \overline U)
\end{align}
does not satisfy $H4')$, because $H4')$ requires $p>\frac{9}{4}$. It becomes
admissible after adding a regularizing term
$
        \varepsilon | \overline U | ^p,
      $ $
        \varepsilon>0,
       $ $
        p>\frac{9}{4}.
$

\item {Saint Venant--Kirchhoff type energy.}
Expressed in terms of the stretch tensor, the Saint Venant--Kirchhoff energy reads
\begin{align}
W_{\rm SVK}(\overline U)
=
\frac{\mu}{4}|\overline U^2-\id_3|^2
+
\frac{\lambda}{8}
\bigl[\tr(\overline U^2-\id_3)\bigr]^2 .
\end{align}
This energy is not globally convex in $\overline U$. Indeed, along
$\overline U=t\,\id_3$, $t>0$,
$
W_{\rm SVK}(t\,\id_3)
=
\left(
\frac{3\mu}{4}
+
\frac{9\lambda}{8}
\right)(t^2-1)^2,
$
whose second derivative is negative for $0<t<1/\sqrt{3}$, provided
$\mu>0$ and $2\mu+3\lambda>0$. Hence the Saint Venant--Kirchhoff energy
does not satisfy H2) or H2') and is not directly covered by the present
framework.

\item A reduced energy independent of $\Cof \overline U$ does not satisfy $H1)-H4)$,
because $H4)$ requires coercivity in $|\Cof \overline U|^r$. Such an energy belongs to
the reduced class $H1')-H4')$.

\item A linear cofactor term
$
         | \Cof \overline U | 
$
is convex, but does not fit $H4)$, because $H4)$ requires
$
        r>1.
$
\item  {
Energies depending explicitly on the inverse of the relative stretch are not
directly covered by the present framework. For instance, one may consider
\begin{align}
W_{\rm stretch}\bigl(x,\overline U,\det\overline U\bigr)
=
\left\langle
\mathbb C(x)
\bigl(\overline U^{-1}-\id_3\bigr),
\overline U^{-1}-\id_3
\right\rangle^{p/2}
+
h(\det\overline U),
\end{align}
whenever $\overline U$ is invertible. Although
$
\overline U^{-1}
=
\frac{(\Cof\overline U)^T}{\det\overline U},
$
the resulting dependence on the lifted variables is not, in general, convex.
Indeed, already in the scalar case the prototype
$
t\mapsto \left(\frac1t-1\right)^2,$ $
 t>0,
$
is not globally convex. Moreover, the inverse-stretch term does not provide
the coercive growth in $|\overline U|^p$ required by H4'), since
$|\overline U^{-1}-\id_3|$ remains bounded along
$\overline U=t\,\id_3$ as $t\to\infty$.
Consequently, such inverse-stretch energies do not satisfy H2)--H4) or
H2')--H4') without additional convexifying and coercive regularization.}

\item  { 
The quadratic Hencky energy is not covered by the present framework. A
natural lifted formulation is obtained by considering the extended-valued
density
\begin{align}
W_{\rm stretch}(\overline U)
=
\begin{cases}
\displaystyle
\mu |\!\log \overline U|^2
+
\frac{\lambda}{2}
\bigl(\tr\log \overline U\bigr)^2,
&
\overline U\in\operatorname{Sym}^{++}(3),
\\[2mm]
+\infty,
&
\text{otherwise}.
\end{cases}
\end{align}
Thus, the occurrence of the matrix logarithm is not by itself an
obstruction. Indeed, in the constrained class considered in the present
paper,
$
\overline U\in\operatorname{Sym}^{+}(3),
$ $
\det\overline U>0,
$
implies
$\overline U\in\operatorname{Sym}^{++}(3)$, so that
$\log\overline U$ is well defined.
The obstruction is instead the lack of global convexity in the lifted
variable $\overline U$. Already along the spherical path
$\overline U=t\,\id_3$, $t>0$, the quadratic Hencky energy contains a
positive multiple of $(\log t)^2$, and
$
\frac{{\rm d}^2}{{\rm d}t^2}(\log t)^2
=
\frac{2(1-\log t)}{t^2},
$
which is negative for $t>e$. Hence the extended Hencky density is not
convex in $\overline U$ and therefore does not satisfy H2) or H2').
Consequently, the standard quadratic Hencky energy is not covered by the
present existence theory.}
 {\item \textbf{Exponentiated logarithmic and exponentiated Hencky energies.}
Exponentiation of the logarithmic strain does not, in general, bring the
corresponding energies into the class considered here. For example, the
isotropic logarithmic energy
\begin{align}
W_{\rm exp}(\overline U)
=
\exp\!\left(k|\log\overline U|^2\right),
\qquad k>0,
\end{align}
may be defined as an extended-valued function on
$\operatorname{Sym}^{++}(3)$, with value $+\infty$ outside this cone.
However, the corresponding elastic energy
$
F\mapsto
\exp\!\left(k|\log U|^2\right),
$ $
U=\sqrt{F^TF},
$
is known to fail global rank-one convexity already in dimension two,
and hence also in dimension three; see
\cite{NeffGhibaLankeit}.
The same distinction is relevant for the exponentiated Hencky energy
introduced in \cite{NeffGhibaLankeit},
\begin{align}
W_{\rm eH}(U)
=
\frac{\mu}{k}
\exp\!\left(k|\dev_n\log U|^2\right)
+
\frac{\kappa}{2\widehat k}
\exp\!\left(
\widehat k[\tr(\log U)]^2
\right).
\end{align}
In dimension two this energy possesses substantially improved convexity
properties: it is rank-one convex for
$k\geq\frac14$ and $\widehat k\geq\frac18$, and suitable parameter
ranges even yield planar polyconvexity \cite{ghiba2015exponentiated,NeffGhibaPoly}. In dimension three, however,
global rank-one convexity is lost; 
$
F\mapsto
\exp\!\left(k|\!\dev_3\log U|^2\right)
$
is not rank-one convex for any $k>0$.}

 {
These results concern rank-one convexity or polyconvexity with respect
to the deformation gradient $F$ and should be distinguished from the
hypothesis used in the present paper, which requires convexity with
respect to the independent lifted variables. In particular, the
logarithmic representations above do not provide the global convexity
in $\overline U$ required by H2) or H2'). Consequently, the
three-dimensional exponentiated logarithmic and exponentiated Hencky
models are not directly covered by the present existence framework.}

\item Another energy not covered by our approach is given by the following energy functional \cite{Neff_Habil04}	\begin{align}\label{minprob}
		\dd\int_{\Omega }\Big\{&
       \mu 
        \left\|\sym(
       \overline{R}^T{\rm D}\varphi
        -
        \id_3)
        \right\|^2
        +{\rm L}_{\rm c}^q \,\|\overline R^T\Curl \overline R\|^q        \Big\}{\rm d}{\rm x}, \qquad 2\leq q<3.	\end{align} This energy is not excluded because of the curvature exponent
$2\leq q<3$, which is compatible with the present assumptions.
The obstruction is instead the stretch term:
$\|\sym(\overline R^T{\rm D}\varphi-\id_3)\|^2$ does not provide
coercive control of the full relative deformation
$\overline U=\overline R^T{\rm D}\varphi$. Hence neither H4) nor H4')
is satisfied; in addition, the quadratic growth $p=2$ lies below the
threshold $p>9/4$ required in the reduced regime.  {Existence results for this type of degenerate quadratic Cosserat energy
rely on generalized Korn inequalities of a different type; see
\cite{neff2002korn}}.
\end{itemize}

\section{Conclusions}

We have developed an existence theory for nonlinear elastic energies formulated through an independent Cosserat microrotation and the relative stretch $ \overline U =  \overline R ^T{\rm D}\varphi$. The curvature variable $  \overline R ^T\operatorname{Curl}  \overline R $ controls the full first-order variation of the rotation field and supplies the strong compactness needed to pass to the limit in the lifted stretch variables.
 Two complementary regimes were treated. In the cofactor-dependent regime, convexity and coercivity are imposed in $(  \overline U ,\operatorname{Cof}  \overline U ,\det  \overline U )$. In the reduced regime, no independent cofactor coercivity is assumed; the threshold $p>9/4$ allows identification of the weak Jacobian limit through the Piola identity and compact Sobolev convergence. The singular determinant dependence then preserves local orientation in the limit.
 For the constrained problem, the closed-cone condition $  \overline R ^T{\rm D}\varphi\in\operatorname{Sym}^{+}(3)$ is weakly stable. Combined with $\det{\rm D}\varphi>0$, it implies $  \overline R =R=\operatorname{polar}({\rm D}\varphi)$ for every admissible pair. The lifted problem is consequently equivalent to a deformation problem whose energy contains the curvature of the polar factor. The result should therefore be interpreted as an existence theorem for stretch-based energies with selective rotational regularization (the couple--stress model \cite{Toupin62,sky2026cosserat}). It does not establish existence for the unregularized first-gradient Biot energy, does not control the variation of the stretch tensor, and does not by itself imply global injectivity. Within these precise limits, the Cosserat lifting provides a robust variational route for several Biot-type, reduced Neo-Hookean-type and cofactor-dependent stretch-polyconvex models.

The choice of the curvature variable $ \overline R^{T}\operatorname{Curl}\overline R$, rather than the full gradient ${\rm D} \overline R$, is deliberate. It preserves the intrinsic Cosserat structure of the model and provides an objective measure of the spatial incompatibility of the microrotation. At the same time, for ${\rm SO}(3)$-valued tensor fields in
${\rm W}^{1,q}(\Omega)$, $q\geq2$, control of
$\operatorname{Curl}\overline R$ controls the full gradient in
${\rm L}^q$ within the ${\rm W}^{1,q}$-admissible class. This choice is also motivated by the perspective of dimensional reduction. Curvature energies expressed through $ \overline R^{T}\operatorname{Curl} \overline R$ admit decompositions that are naturally compatible with the tangential and transverse structure of thin domains and therefore provide a suitable starting point for the rigorous derivation of lower-dimensional Cosserat theories for shells, plates, and rods. Thus, the present formulation is intended not only as an existence framework for three-dimensional models, but also as a basis for subsequent dimension-reduction results preserving the geometric and variational structure of the rotational curvature.
 
 Nevertheless, if one considers a broader class of energies depending
directly on the full gradient ${\rm D} \overline R$, the present analysis suggests
a corresponding multifield extension. The Cosserat-polyconvexity
condition may then be replaced by a multifield polyconvexity assumption
requiring the stored-energy density to admit a convex representation
in terms of all minors of the combined gradient
\[
\mathbb G(\varphi, \overline R)
:=
{\rm D}\bigl(\varphi,\operatorname{vec} \overline R\bigr)
=
\begin{pmatrix}
{\rm D}\varphi\\
{\rm D}(\operatorname{vec} \overline R)
\end{pmatrix}
\in\mathbb R^{12\times3},\qquad 
\begin{pmatrix}
{\rm D}\varphi\\
{\rm D}(\operatorname{vec} \overline R)
\end{pmatrix}\approx
\begin{pmatrix}
{\rm D}\varphi\\
\Curl \overline R,
\end{pmatrix}
\]
including the corresponding mixed minors, see Mariano and Modica \cite{Mariano08a}. 
Although, for ${\rm SO}(3)$-valued fields, the Curl--Grad relation
provides control of the full gradient ${\rm D}\overline R$ through
$\Curl\overline R$, the mixed minors appearing in the multifield
polyconvexity framework of Mariano and Modica are minors of the full
combined gradient
$
{\rm D}\bigl(\varphi,\operatorname{vec}\overline R\bigr)
$. 
They cannot be obtained by a literal substitution of
${\rm D}(\operatorname{vec}\overline R)$ by $\Curl\overline R$.
A corresponding formulation in terms of the Curl curvature would
therefore require a separate analysis of the associated mixed minors.

However, under suitable coercivity
and integrability assumptions guaranteeing weak compactness and weak
continuity of these minors, the direct-method arguments developed in the present paper
can be adapted to this enlarged class of energies.
This extension would incorporate into the multifield polyconvex framework of Mariano and Modica \cite{Mariano08a} the orientation-preserving, weak-closure, and polar-recovery mechanisms established in the present work. 
Building on the approach developed in the present paper, a work in preparation aims to compare and combine, in a complementary manner, the multifield existence framework of Mariano and Modica \cite{Mariano08a} with the finite-strain micromorphic existence theory of Neff \cite{Neff_Habil04,NeffMicromorphic2006}.  

\begin{footnotesize}
	\bibliographystyle{plain} 

 \end{footnotesize}

\begin{footnotesize}
\appendix
\section{Further admissible energies and boundary cases}\label{app:further-examples}\setcounter{equation}{0}

This appendix records additional energy classes covered by the hypotheses and several instructive models that require further modification. The notation and the assumptions on the volumetric term $h$ are those of Section~\ref{s6}.

\subsection{Additional reduced determinant-dependent energies}

The following models satisfy H1')--H4') under the exponent assumptions stated in Section~\ref{s6}.

\begin{itemize}

\item \textbf{Pure $p$-Biot energy.}
One may also consider
\begin{align}
        W_{\rm stretch}\bigl(x,\overline U,\det \overline U\bigr)
        =
        \mu  | \overline U-\id_3 | ^p
        +
        h(\det \overline U),
        \qquad
        \mu>0,
        \qquad
        p>\frac{9}{4}.
\end{align}
This is the simplest reduced Cosserat-Biot model satisfying $H1')-H4')$.

\item \textbf{Matrix power-law stretch energies inspired by Ogden-type growth.}
Matrix power-law energies inspired by Ogden-type growth may be written
in the reduced decoupled form
\begin{align}
        W_{\rm stretch}\bigl(x,\overline U,\det \overline U\bigr)
        =
        \sum_{i=1}^N c_i  | \overline  U | ^{m_i}
        +
        h(\det \overline U),
\end{align}
or, in a coupled Biot-shifted form,
\begin{align}
        W_{\rm stretch}\bigl(x,\overline U,\det \overline U\bigr)
        =
        \sum_{i=1}^N c_i  | \overline U-\id_3 | ^{m_i}
        +
        h(\det \overline U),
\end{align}
where
$
        c_i>0,
$ $
        m_i\geq1.
$
The hypotheses $H1')-H4')$ are satisfied provided that at least one exponent
gives the required growth
$
        \max_i m_i=p>\frac{9}{4}.
$

\item \textbf{Schatten-norm Ogden-inspired decoupled energies.}
A coordinate-free  version is
\begin{align}
        W_{\rm stretch}\bigl(x,\overline U,\det \overline U\bigr)
        =
        \sum_{i=1}^N c_i  | \overline U | _{S_{m_i}}^{m_i}
        +
        h(\det \overline U),
\end{align}
where $ | \cdot | _{S_m}$ denotes the Schatten norm. Since
$
        \overline U\mapsto  |\overline  U | _{S_m}^{m}
$
is convex for $m\geq1$, this class satisfies $H1')-H4')$ whenever
$
        c_i>0,
         $ $
        \max_i m_i=p>\frac{9}{4}.
$

\item \textbf{Fung-type exponential  energies.}
A convex exponential stiffening model is
\begin{align}
        W_{\rm stretch}\bigl(x,\overline U,\det \overline U\bigr)
        =
        a\left[
        \exp\left(\eta | \overline U-\id_3 | ^p\right)-1
        \right]
        +
        h(\det\overline  U),
\end{align}
with
$
        a,\eta>0,
           $ $
        p>\frac{9}{4}.
$
The exponential term is convex because it is the composition of the increasing
convex function $t\mapsto e^{\eta t}-1$ with the convex function
$\overline U\mapsto |\overline U-\id_3|^p$.

\item \textbf{Power-law stiffening energies.}
A flexible stretch-stiffening class is
\begin{align}
        W_{\rm stretch}\bigl(x,\overline U,\det \overline U\bigr)
        =
        a | \overline U-\id_3 | ^p
        +
        \sum_{j=1}^N a_j  | \overline U-\id_3 | ^{p_j}
        +
        h(\det\overline  U),
\end{align}
where
$
        a>0,
            $ $
        p>\frac{9}{4},
       $ $
        a_j\geq0,
         $ $
        p_j\geq1.
$
The first term provides the required coercivity in $\overline U$.

\item \textbf{Inhomogeneous preferred stretch energies.}
Let
$
        U_0\in {\rm L}^\infty(\Omega;\mathbb R^{3\times3})
$
be a prescribed preferred stretch and let $a(x)$ be measurable with
$
        0<a_0\leq a(x)\leq a_1<+\infty .
$
Then
\begin{align}
        W_{\rm stretch}\bigl(x,\overline U,\det \overline U\bigr)
        =
        a(x) | \overline U-U_0(x) | ^p
        +
        h(\det \overline U),
        \qquad
        p>\frac{9}{4},
\end{align}
satisfies $H1')-H4')$.

\item \textbf{Deviatoric-spherical stretch energies.}
Since $\overline U\mapsto \dev \ \overline U$ and $\overline U\mapsto \tr \overline U$ are linear, the density
\begin{align}
        W_{\rm stretch}\bigl(x,\overline U,\det \overline U\bigr)
        =
        \mu | \dev (\overline U-\id_3) | ^p
        +
        \kappa | \tr (\overline U-\id_3)| ^p
        +
        h(\det \overline U)
\end{align}
is convex in $\overline U$. If
$
        \mu,\kappa>0,
        \      $ $
        p>\frac{9}{4},
$
then the full norm $|\overline U|^p$ is controlled and $H1')-H4')$ are satisfied.

\item \textbf{Fibre-reinforced stretch energies.}
Let $a_1,\ldots,a_N\in\mathbb S^2$ be preferred material directions. The
density
\begin{align}
        W_{\rm stretch}\bigl(x,\overline U,\det U\bigr)
        =
        \mu  | \overline U-\id_3 | ^p
        +
        \sum_{i=1}^N c_i  | \overline Ua_i | ^{m_i}
        +
        h(\det \overline U)
\end{align}
satisfies $H1')-H4')$ if
$
        \mu>0,
       $ $
        p>\frac{9}{4},
        $ $
        c_i\geq0,
          $ $
        m_i\geq1.
$
Threshold-type fibre activation is also admissible:
\begin{align}
        W_{\rm stretch}\bigl(x,\overline U,\det \overline U\bigr)
        =
        \mu | \overline U-\id_3 | ^p
        +
        \sum_{i=1}^N
        c_i\bigl( | \overline Ua_i | -\ell_i\bigr)_+^{m_i}
        +
        h(\det\overline  U).
\end{align}

\end{itemize}

\subsection{Additional cofactor-dependent energies}

The following models satisfy H1)--H4) under the corresponding cofactor-growth assumptions.

\begin{itemize}

\item \textbf{Cofactor-regularized Biot-type energy.}
Although the Biot energy is naturally a reduced stretch energy, one may add a
cofactor regularization if one wants to work in the full cofactor-dependent
hypotheses $H1)-H4)$:
\begin{align}
        W_{\rm stretch}\bigl(x,\overline U,\Cof \overline U,\det \overline U\bigr)
        =
        \mu  | \overline U-\id_3 | ^2
        +
        \varepsilon  | \overline U-\id_3 | ^p
        +
        b | \Cof \overline U  -\id_3| ^r
        +
        h(\det \overline U),
\end{align}
where
$
        \mu\geq0,
            $ $
        \varepsilon>0,
            $ $
        b>0,
              $ $
        p\geq2,
          $ $
        r>1,
           $ $
        \frac1p+\frac1r<\frac43.
$
If $b=0$, the model belongs to the reduced class, not to $H1)-H4)$.

\item \textbf{Cofactor-augmented Ogden-inspired energy.}
If a model includes a genuine cofactor contribution, one may take
\begin{align}
        W_{\rm stretch}\bigl(x,\overline U,\Cof \overline U,\det \overline U\bigr)
        =
        \sum_{i=1}^N c_i  | \overline U | ^{m_i}
        +
        b | \Cof \overline U | ^r
        +
        h(\det \overline U),
\end{align}
or
\begin{align}
        W_{\rm stretch}\bigl(x,\overline U,\Cof \overline U,\det\overline  U\bigr)
        =
        \sum_{i=1}^N c_i  | \overline U-\id_3 | ^{m_i}
        +
        b | \Cof \overline U-\id_3 | ^r
        +
        h(\det \overline U),
\end{align}
where
$
        c_i>0,
             $ $
    m_i>1,
         $ $
        b>0,
         $ $
        r>1,
      $ $
        \frac1p+\frac1r<\frac43.
$

\item \textbf{Area-fibre or cofactor-fibre  energies.}
The cofactor variable is natural for modelling area-type responses. Let
$a_1,\ldots,a_N\in\mathbb S^2$. One may consider
\begin{align}
        W_{\rm stretch}\bigl(x,\overline U,\Cof \overline U,\det \overline U\bigr)
        =
        g(x,\overline  U)
        +
        \sum_{i=1}^N b_i | \Cof \overline U\,a_i | ^r
        +
        h(\det \overline U),
\end{align}
where $g(x,\cdot)$ is convex and satisfies
$
        g(x,\overline U)\geq c|\overline U|^p-C.
$
This satisfies $H1)-H4)$ if the directions $a_i$ control the full cofactor,
namely if
$
        \sum_{i=1}^N  | B\,a_i | ^r
        \geq
        c_{\rm cof} | B | ^r
$
for all $B\in  \mathbb{R}^{3\times 3}$, and if
$
        p\geq2,
      $ $
        r>1,
         $ $
        \frac1p+\frac1r<\frac43.
$
If the fibre directions do not control the full cofactor, one should add an
explicit term
$
        b_0 | \Cof \overline U | ^r,
        \  b_0>0.
$

\item \textbf{General separable Cosserat-polyconvex class.}
A broad cofactor-dependent class is
\begin{align}
        W_{\rm stretch}\bigl(x,\overline U,\Cof \overline U,\det \overline U\bigr)
        =
        g_1(x,\overline U)
        +
        g_2(x,\Cof \overline U)
        +
        h(\det \overline U),
\end{align}
where $g_1(x,\cdot)$, $g_2(x,\cdot)$, and $h$ are convex, measurable in $x$,
and satisfy
$
        g_1(x,\overline U)\geq c_{\overline U} | \overline U | ^p-C_{\overline U},
$
$
        g_2(x,\Cof \overline U)\geq c_{\rm cof} | \Cof \overline U | ^r-C_{\rm cof},
$
$
        h(\det \overline U)
        \geq
        c_d(\det \overline U)^s
        +
        \gamma(\det \overline U)^{-\beta}
        -
        C_d,
        \  
        \det \overline U>0.
$
The exponents must satisfy
$
        p\geq2,
         $ $
        r>1,
        $ $
        s>1,
      $ $
        \frac1p+\frac1r<\frac43.
$

\end{itemize}

\end{footnotesize}  

\end{document}